\documentclass[11pt]{amsart}
\usepackage{amsmath}
\usepackage{amssymb}
\usepackage{amsthm}

\usepackage[utf8]{inputenc}
\usepackage[]{graphicx}                
\usepackage{color}
\usepackage{multicol}
\usepackage[T1]{fontenc}      
\usepackage{nicefrac}
\usepackage{wasysym}
\usepackage{pifont}

\newtheorem{thm}{Theorem}[section]
\newtheorem{defn}[thm]{Definition}

\newtheorem{prop}[thm]{Proposition}
\newtheorem{rem}[thm]{Remark}

\newtheorem{example}[thm]{Example}
\newtheorem{exercise}[thm]{Exercise}
\newtheorem{notation}[thm]{Notation}

\DeclareMathOperator{\supp}{supp}
\DeclareMathOperator{\XL}{{\color{blue}{XL}}}

\begin{document}
\title{Special Functions from a Complex Viewpoint}
\author[Henrik L.\ Pedersen]{Henrik L.\ Pedersen\\Department of Mathematical Sciences\\University of Copenhagen\\Denmark}

  \date{OPSFA Summer School, July 29th -- August 2nd 2024}

\maketitle

\tableofcontents
\section*{Introduction. Overview of the course}
This document contains the lecture notes for a mini-course on special functions from a complex viewpoint given at the OPSFA Summer School
 \emph{OPSFA-S10 2024}, in the period July 29th -- August 2nd, 2024. The summer school was hosted by
Section of Mathematics -- Luciano Modica at the
Uninettuno University.

\emph{Going from the real to the complex and back again!}
As we focus on applications of complex analysis in the solution of problems for special functions on the real line, this sentence expresses the main theme of the course. Also some real variable methods are described. We shall give several examples of the use of the techniques in connection with special functions.

The course consists of five lectures, briefly outline in the next paragraph. Occasionally, a (partial) proof is given as illustration of how the techniques come into play. In the introduction to each lecture we mention relevant references to the theory. Some exercises are included at the end of each lecture.

In the first lecture focus is on classes of holomorphic functions in the upper half-plane as well as some background results from complex analysis. This theme is continued in the second lecture, where Stietjes functions are investigated. The third lecture contains a discussion of inverses of Euler's gamma function and furthermore so-called  multi or higher order gamma functions introduced by Barnes. After a brief introduction we consider the behaviour of remainders in asymptotic expansions of the logarithm of some of these functions. In the fourth lecture the main theme is the class of completely monotonic functions and related classes. Several examples illustrate the main tools in use; including complex methods. In the fifth lecture we deal with so-called generalized Bernstein functions of positive order including the relation to the class of generalized Stieltjes functions of positive order. Examples related to Euler's gamma function are given.

Main references are mentioned in the beginning of each lecture and I appologise for any references left out.

\section{First lecture: Motivation. Functions in the upper half-plane}

After a brief introduction to Euler's Gamma function (see \cite{Artin} and \cite{RSFM}) the theme of the course is motivated by describing a problem about monotonicity properties on the positive line of a function related to the volume of the unit ball in $\mathbb R^n$ and Euler's Gamma function.

Positivity is an important feature in our investigations and we present maximum principles for holomorphic and (sub)harmonic functions in unbounded domains of the complex plane. References: \cite{Koosis}.

Thereafter the class of Nevanlinna--Pick functions in the upper half-plane is introduced. These functions appear in subjects such as operator theory, functional analysis and in particular in the classical moment problem on the real line. The fundamental properties of this class are presented.
References: \cite{A}, \cite{D}, \cite[Appendix A.2]{Smy}.

 \subsection{Motivation, Euler's Gamma function}


 The Gamma function, denoted by $\Gamma$, was introduced by Leonhard Euler and may be defined as
 $$
 \Gamma(x)=\int_0^{\infty}e^{-t}t^{x-1}\, dt
 $$
 for $x>0$. The graph of $\Gamma$ is shown in Figure \ref{fig:gamma-plot}.
\begin{figure}[htb]
  \includegraphics[width=0.75\textwidth]{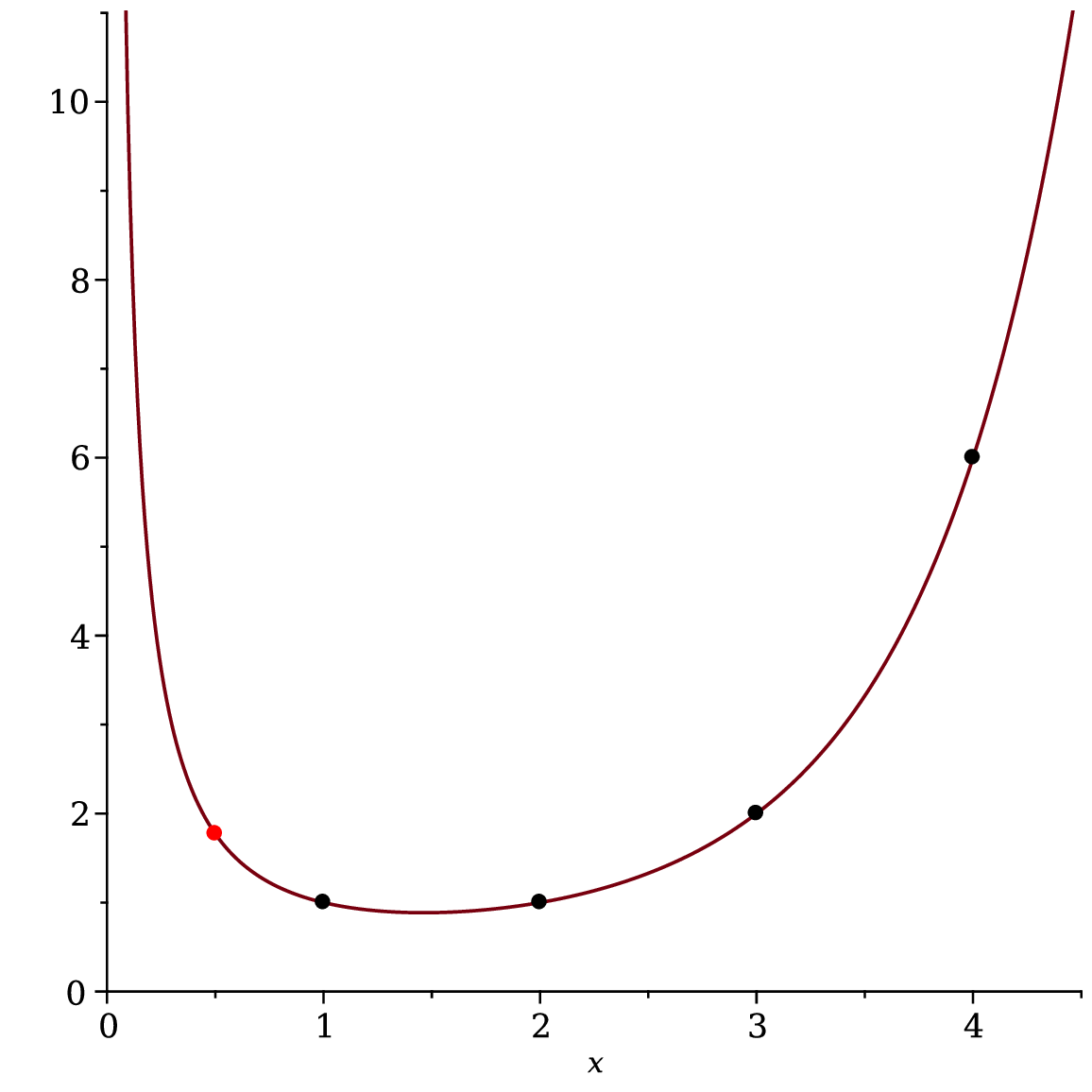}
  \caption{Graph of $\Gamma$ on $(0,\infty)$}
\label{fig:gamma-plot}
  \end{figure}
Let us highlight a few formulae that are easily verified by computation. The first one follows by partial integration and is the famous \emph{functional equation of the Gamma function}
 $$
 \Gamma(x+1)=x\Gamma(x),
 $$
 and its particular case $\Gamma(n+1)=n!$ for $n\in \{ 0,1,\ldots\}$.
As a curiosity let us remark that a change of variable shows that $\Gamma(\nicefrac{1}{2})=\sqrt{\pi}$. (This is illustrated by a red dot in Figure \ref{fig:gamma-plot}.)

The Gamma function can be extended to the right half-plane via the integral representation,
$$
\Gamma(z)=\int_0^{\infty}e^{-t}t^{z-1}\, dt,\quad \Re z>0.
$$
The Weierstra\ss\ infinite product representation is valid in the entire complex plane and is as follows:
$$
\frac{1}{\Gamma(z)}=ze^{\gamma z}\prod_{n=1}^{\infty}\left(1+\nicefrac{z}{n}\right)e^{-\nicefrac{z}{n}},\quad z\in \mathbb  C.
$$
Here $\gamma$ is the so-called Euler's constant chosen such that the right hand side equals $1$ for $z=1$.
In the complex plane cut along the negative real line, $\mathbb C\setminus (-\infty, 0]$ we define
\begin{equation}
\label{eq:logGamma}
\log \Gamma(z)=-\gamma z-\log z-\sum_{n=1}^{\infty}\left(\log (1+\nicefrac{z}{n})-\nicefrac{z}{n}\right).
\end{equation}
Notice that $\log$ on the right hand side means the principal branch of the logarithm, being real on the positive line. The $\log \Gamma$ thus defined is a holomorphic branch of the logarithm of $\Gamma$, but at $z$ it is not necessarily equal to the principal logarithm of the complex number $\Gamma(z)$.

Let us turn to the behaviour of $\Gamma(x)$ as $x$ tends to $\infty$ along the real line. The classical and well-known \emph{Stirling formula} can be expressed as the relation
 $$
 \Gamma(x)=\sqrt{2\pi}x^{x-\nicefrac{1}{2}}e^{-x}e^{\mu(x)}\quad \text{for}\ x>0, \quad \text{with}\ 0<\mu(x)<\nicefrac{1}{12x}.
 $$
As a consequence we have
 $$
 \nicefrac{\log \Gamma(x+1)}{x\log x}\to 1,\quad \text{as}\ x\to \infty.
 $$
 The question about the way in which the left hand side approaches $1$ is of course a question about monotonicity properties of the function
   \begin{equation}\label{eq:logGammaratio}
   f(x)=\frac{\log \Gamma(x+1)}{x\log x}.
   \end{equation}
   It turns out that its derivatives of odd order are positive, and the derivatives of even order are negative! The main ideas behind obtaining information on the sign changes of the derivatives of $f$ from $f$ itself \emph{are definitely not} to differentiate $f$ directly but rather to extend the function to the upper half-plane, and to use suitable maximum principles, thereby obtaining an integral representation. Here is the plan:
   \begin{enumerate}
    \item Extend $f$ to a holomorphic function in the complex plane cut along the non-positive real line.
    \item Consider its imaginary part and show that it is a positive harmonic function in the upper half-plane.
    \item From the positivity find an integral representation of the imaginary part using its boundary behaviour on the negative real line, and then find an integral representation of the function itself.
    \item Differentiate this integral representation and obtain the sign pattern of the derivatives.
   \end{enumerate}
   These steps are illustrated by the four plots in Figure \ref{fig:fourplots}.
\begin{figure}
 \begin{center}
 \begin{tabular}{cc}
 \includegraphics[scale=0.25]{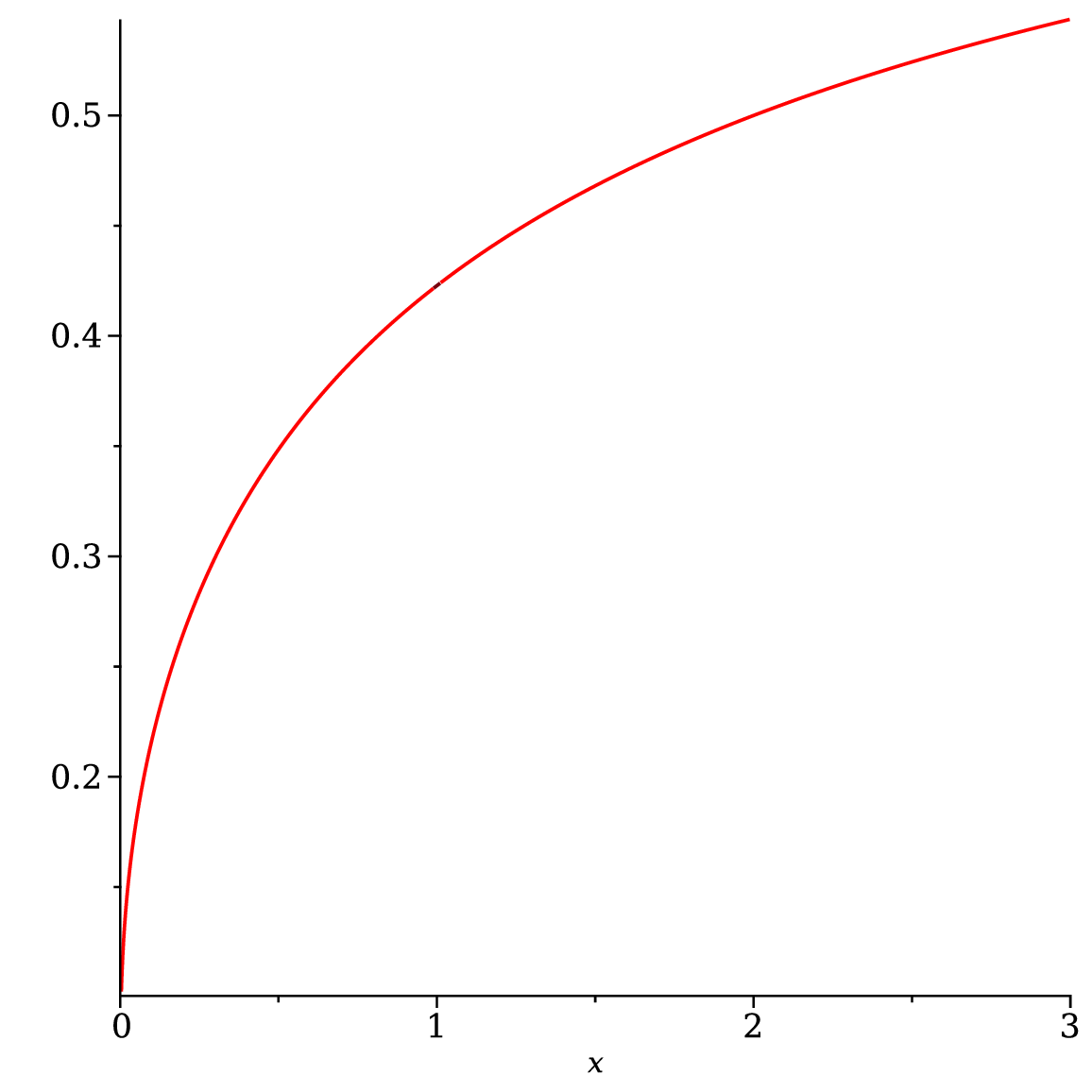}&
 \includegraphics[scale=0.25]{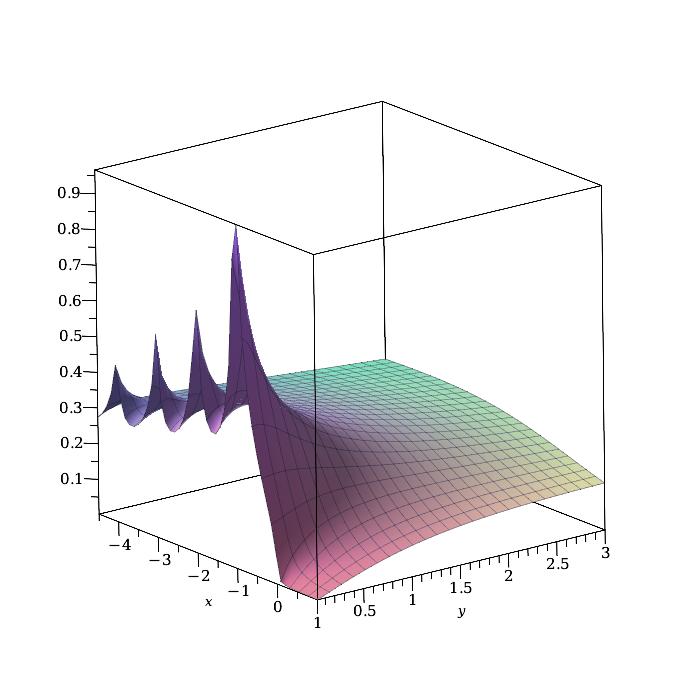}\\
 \includegraphics[scale=0.25]{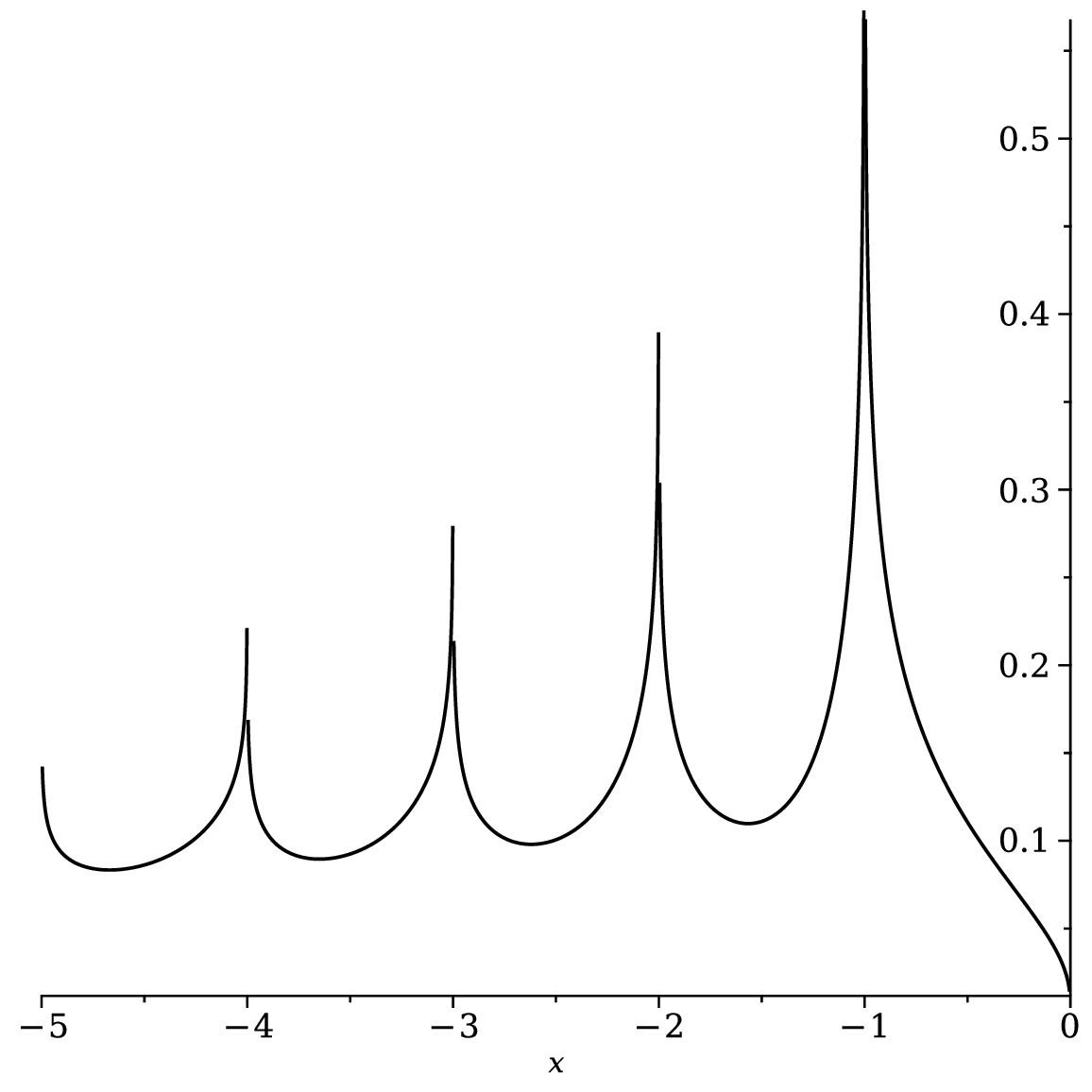}&
 \includegraphics[scale=0.25]{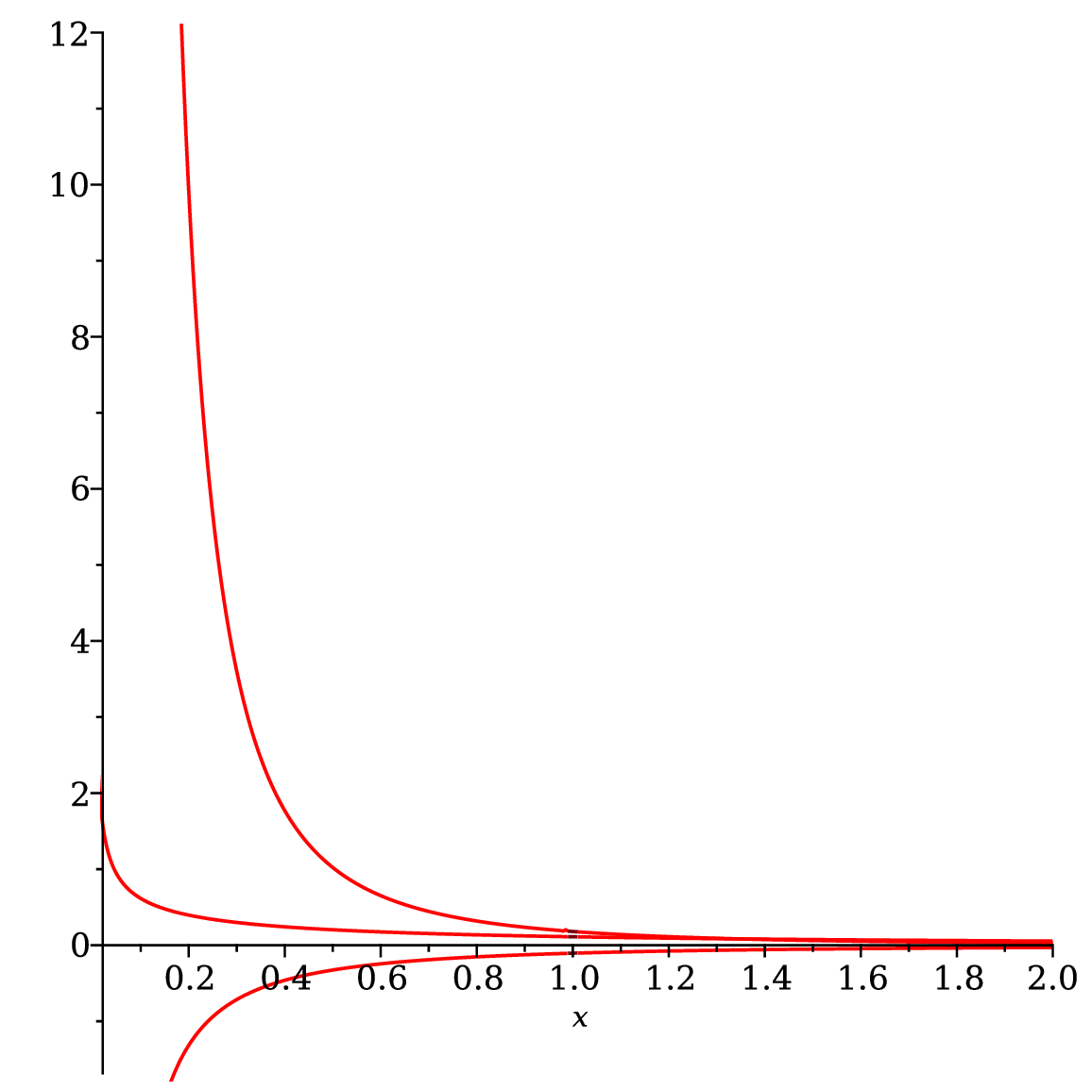}
 \end{tabular}
 \end{center}
 \caption{The function on the positive line; its imaginary part in the upper half-plane; the boundary behaviour on the negative axis; the derivatives on the positive axis.}
 \label{fig:fourplots}

 \end{figure}
 \begin{example} The function above is also closely related to the investigation of properties of the volume $\Omega_n$ of the unit ball in $\mathbb R^n$. This volume can be expressed via the Gamma function as follows
  $$
  \Omega_n=\frac{\pi^{\nicefrac{n}{2}}}{\Gamma(\nicefrac{n}{2}+1)}.
  $$
The asymptotic behaviour of $\Omega_n$ is easily determined using Stirling's formula and we obtain
$$
\Omega_n^{\nicefrac{1}{n\log n}}\to \sqrt{e}.
$$
  The questions arose if this sequence is decreasing or if it is logarithmically convex. The answers to these questions can be obtained from the representation
   $$
   \Omega_n^{\nicefrac{1}{n\log n}}=\int_0^1t^{n-2}\, d\mu(t),\quad n\geq 2,
   $$
   where $\mu$ is some positive measure on $[0,1]$ having moments of all orders. This can be expressed by saying that the sequence $\Omega_n^{\nicefrac{1}{n\log n}}$ is a Hausdorff moment sequence, and such sequences are known to be decreasing and logarithmically convex. The representation of $\Omega_n^{\nicefrac{1}{n\log n}}$ as a moment sequence is obtained using properties of a slightly modified version of the function $f$ above, since
   $\log \Omega_n^{\nicefrac{1}{n\log n}}=\nicefrac{\log \pi}{2\log n}-\nicefrac{\log \Gamma(\nicefrac{n}{2}+1)}{n\log n}$.

\end{example}

\subsection{Preliminaries from basic complex analysis}
In this section focus is on maximum principles for holomorphic and subharmonic functions.
  \begin{notation}{\quad }
   \begin{enumerate}
    \item By $\mathcal H(\Omega)$ we denote the class of holomorphic functions defined in the open subset $\Omega$ of $\mathbb C$.
    \item A domain is an open and connected subset of the complex plane.
    \item The open upper half-plane is denoted by $\mathbb H$.
   \end{enumerate}

  \end{notation}

 \begin{prop}
  [Local version of the ordinary maximum principle]
  If $\Omega$ is a domain in $\mathbb C$ and if $f\in \mathcal H(\Omega)$ has a local maximum in $\Omega$ then $f$ is constant.
 \end{prop}

 \begin{prop}
  [A maximum principle in a bounded domain] Let $\Omega$ be a bounded domain in $\mathbb C$. If $f\in \mathcal H(\Omega)$ has a continuous extension to $\overline{\Omega}$ then $|f(z)|\leq \max_{w\in \partial \Omega}|f(w)|$.
 \end{prop}

 There are also maximum principles for (sub)harmonic functions. Maximum principles, sometimes also called Phragm\'en-Lindel\"of principles, in unbounded domains are a bit more delicate. (The function $f(z)=\sin z$ is bounded on $\mathbb R$ but not in the upper half-plane.)


\begin{defn}
 Let $\Omega$ be a domain. A function $u:\Omega\to [-\infty,\infty)$ is \emph{subharmonic} if
 \begin{enumerate}
  \item $u$ is upper semicontinuous: $u^{-1}([-\infty,a))$ is open for all $a\in \mathbb R$;
  \item $u$ satisfies a local submean inequality: for all $z\in \Omega$ there is $\rho>0$ such that for all $r\in [0,\rho)$,
  $$u(z)\leq \frac{1}{2\pi}\int_0^{2\pi}u(z+re^{it})\, dt.$$
 \end{enumerate}
\end{defn}

\begin{example}
 Take any $f\in \mathcal H(\Omega)$. Then $u(z)=\log |f(z)|$ is subharmonic.
\end{example}


 \begin{prop}[Maximum principle for subharmonic functions]
  Let $\Omega$ be a bounded domain in $\mathbb C$, and let $u$ be subharmonic in $\Omega$.
  \begin{enumerate}
   \item If $u$ attains a global maximum on $\Omega$ then $u$ is constant.
   \item If, for each $\zeta\in \partial \Omega$,
  $
  \limsup_{z\to \zeta,\ z\in \Omega}u(z)\leq 0
  $
  then $u\leq 0$ in $\Omega$.
  \end{enumerate}

 \end{prop}
\begin{proof}
(1): If $u$ attains its maximum value $M$ in $\Omega$ we look at
$$
A=\{z\in \Omega\, | \, u(z)=M\}, \quad  B=\{z\in \Omega\, | \, u(z)<M\}.
$$
The set $B$ is open since $u$ is upper semicontinuous. The set $A$ is also open since the mean value property forces $u$ to be $\equiv M$ in a small neighbourhood of any point in $A$. Since $A\neq \emptyset$, $A\cup B=\Omega$, $A\cap B=\emptyset$ and $\Omega$ is connected we find that $A=\Omega$

(2): Extend $u$ to the boundary of $\Omega$ by $u(\zeta)=\limsup_{z\to \zeta, z\in \Omega}u(z)$, for $\zeta\in \partial \Omega$. Then we have an upper semicontinuous function $u:\overline{\Omega}\to [-\infty,\infty)$. Such a function attains its maximum at some $z_0$ in $\overline{\Omega}$. If $z_0\in \partial \Omega$, $M\leq 0$. If $z_0\in \Omega$ then $u$ must be constant in $\Omega$ and is hence non-positive.
\end{proof}

 \begin{prop}[Maximum principle in unbounded domains]
  Let $\Omega\subsetneq $ be a domain in $\mathbb C$, and let $u:\Omega \to [-\infty,\infty)$ be subharmonic and bounded from above. If there is $m\in \mathbb R$ such that for each $\zeta\in \partial \Omega$,
  $$
  \limsup_{z\to \zeta,\ z\in \Omega}u(z)\leq m
  $$
  then $u\leq m$ in $\Omega$.
 \end{prop}
 \begin{proof}
  We may assume that $m=0$ and $0\in \partial \Omega$. Let $\eta>0$ be fixed. There must exist $\rho>0$ such that $u(z)<\eta$ for all $z\in \Omega\cap B(0,\rho)$ (otherwise, $\limsup_{z\in \Omega, z\to 0} u(z)\geq \eta$, contrary to our assumption).

  Define $\Omega_{\rho}=\Omega\setminus \overline{B(0,\rho)}$ and notice that $\partial \Omega_{\rho}\subseteq \partial \Omega\cup(\Omega\cap \partial B(0,\rho))$. Furthermore, we have
  $$
  \limsup_{z\in \Omega_{\rho}, z\to\zeta}u(z)\leq \eta,\quad \text{for all}\ \zeta \in \partial \Omega_{\rho}.
  $$
  Let $\epsilon >0$ and define $v_{\epsilon}(z)=u(z)-\epsilon \log |z|$. We have
  $$
  v_{\epsilon}(z)\leq u(z)+\epsilon \log (1/\rho)\leq u(z)+\eta
  $$
  for $z\in \Omega_{\rho}$,
  when $\epsilon$ is sufficiently small. Furthermore,
  $$
  \limsup_{z\in \Omega_{\rho}, z\to\zeta}v_{\epsilon}(z)\leq 2\eta
  $$
  for $\zeta\in \partial \Omega_{\rho}$.

  Let now $z_0\in\Omega_{\rho}$. We show that $u(z_0)\leq 2\eta$: Since $u$ is bounded form above, say by $M$, we have
  $$
  v_{\epsilon}(z)\leq M -\epsilon\log|z|=M -\epsilon\log R\leq 0,\quad z\in \Omega_{\rho}\cap \partial B(0,R)
  $$
  for $R$ large enough. Define now $\Omega_{\rho,R}=\Omega_{\rho}\cap B(0,R)$ and observe that
  $$
  \limsup_{z\in \Omega_{\rho,R}, z\to\zeta}v_{\epsilon}(z)\leq 2\eta \quad \text{for}\ \zeta\in \partial \Omega_{\rho, R}
  $$
  The next step is to use the ordinary maximum principle on (the components of) the open and bounded set $\Omega_{\rho,R}$ to obtain $v_{\epsilon}\leq 2\eta$ in $\Omega_{\rho,R}$. In particular, $u(z_0)\leq 2\eta+\epsilon\log [z_0|$. Finally, let $\epsilon$ tend to zero.
\end{proof}



 \begin{prop}[Growth times opening theorem]
  Let $0<\gamma<\pi$ and $u$ be subharmonic in $S=\{z\in \mathbb C\, |\, |\arg z|\leq \gamma\}$, ($\arg$ denotes the principal argument) and suppose that
  $$
  u(z)\leq C+A|z|^{\alpha}, \quad z\in S,
  $$
  where $0\leq \alpha 2\gamma <\pi$.

  If, for all $\zeta\in \partial S$, $\limsup_{z\to \zeta,z\in S}u(z)\leq 0$ then $u\leq 0$ in $S$.
 \end{prop}
 The name of the result refers to the fact that the opening of the sector $S$ is $2\gamma$.
\begin{prop}[A Phragm\'en-Lindel\"of theorem in the upper half-plane]
  Let $u$ be a subharmonic function in $\mathbb H$ such that
  \begin{enumerate}
  \item $u(z)\leq A|z|+o(|z|)$ for large $|z|$, and
  \item for each $x\in \mathbb R$ we have
  $$
  \limsup_{z\to x,z\in \mathbb H}u(z)\leq 0.
  $$
  \end{enumerate}
 Then $u(z)\leq A\Im z$ for $z\in \mathbb H$.
 \end{prop}


\subsection{Nevanlinna--Pick functions}

We denote by $M^+(X)$ the set of positive Radon measures on the locally compact Hausdorff space $X$.

\begin{defn}[Vague topology]
 Let $X$ be a locally compact Hausdorff space.
 The vague topology is the coarsest topology on $M^+(X)$ for which the mappings $\mu\mapsto \int_Xf\, d\mu$ is continuous for all $f\in C_c(X)$. Here $C_c(X)$ denotes the continuous functions on $X$ of compact support.
\end{defn}
\begin{rem}
 {Remark}
 When $X$ is compact the vague topology is the weak-star topology.
\end{rem}
\begin{prop}
 For any $a>0$ the set
 $$
 \{ \mu\in M^+(X)\, |\, \mu(X)\leq a\}
 $$
 is vaguely compact.

 When $X$ is compact, $\{ \mu\in M^+(X)\, |\, \mu(X)= a\}$ is weak-star compact.
\end{prop}

\begin{defn}[Nevanlinna--Pick function]
A function $f\in \mathcal H(\mathbb H)$ is a Nevanlinna--Pick function if  $f(\mathbb H)\subseteq \overline{\mathbb  H}$. The class of Nevanlinna--Pick functions is denoted by $\mathcal N$.
\end{defn}

\begin{example}{The following functions are Nevanlinna--Pick functions.}
\begin{enumerate}
 \item $f(z)=az+b$, where  $a\geq 0$ and $b\in \mathbb R$,
 \item $f(z)=i$,
 \item $f(z)=\log z$, where $\log$ denotes the principal logarithm,
 \item $f(z)=\nicefrac{az+b}{cz+d}$, where $a,b,c,d\in \mathbb R$ and $ad-bc>0$.
 \end{enumerate}

\end{example}


 We proceed to show that any Nevanlinna--Pick function admits an integral representation. We first give a similar result for the so-called {Herglotz functions} in the unit disk and then use conformal mapping to transform the unit disk to the upper half-plane.
 \begin{defn}[Herglotz function]
  A function $f\in \mathcal H(B(0,1))$ is a Herglotz function if $\Re f(z)\geq 0$ for $z\in B(0,1)$.
 \end{defn}
\begin{prop}
 [Integral representation of Herglotz functions]
 The formula
 $$
 f(z)=ib+\int_{\partial B(0,1)}\frac{s+z}{s-z}\, d\mu(s),\ z\in B(0,1)
 $$
 gives a one-to-one correspondence between the class of Herglotz functions and pairs $(b,\mu)\in \mathbb R\times M^+(\partial B(0,1))$. Furthermore ($m$ denoting the Haar-measure on the unit circle),
 $$
 b=\Im f(0),\quad \mu=\lim_{r\to 1-}\Re f(rs)dm(s) \ \text{weak-star.}
 $$
\end{prop}
\begin{proof}
 Let $f$ be a Herglotz function and write it as
 $$
 f(z)=\sum_{n=0}^{\infty}a_nz^n, \ z\in B(0,1).
 $$
 Let $\alpha=\Re a_0$ and $\beta=\Im a_0$, and consider $g(z)=\Re f(z)$.
 We have
 $$
 g(re^{i\theta})=\alpha+\frac{1}{2}\sum_{n=1}^{\infty}\left(a_nr^ne^{in\theta}+\overline{a_n}r^ne^{-in\theta}\right).
 $$
 This is a Fourier series and we obtain
 $$
 \alpha=\frac{1}{2\pi}\int_0^{2\pi}g(re^{i\theta})\, d\theta,\quad \frac{1}{2}a_nr^n=\frac{1}{2\pi}\int_0^{2\pi}g(re^{i\theta})e^{-in\theta}\, d\theta
 $$
 for $n\geq 1$. For $0<r<1$ and $z\in B(0,1)$ we have
 \begin{align*}
 f(rz)&=\sum_{n=0}^{\infty}a_nr^nz^n\\
 &=\alpha+i\beta+\sum_{n=1}^{\infty}\frac{2}{2\pi}\int_0^{2\pi}g(re^{i\theta})e^{-in\theta}\, d\theta \, z^n\\
 &=i\beta+\frac{1}{2\pi}\int_0^{2\pi}g(re^{i\theta})\left(1+2\sum_{n=1}^{\infty}z^ne^{-in\theta}\right)\, d\theta\\
 &=i\beta+\frac{1}{2\pi}\int_0^{2\pi}g(re^{i\theta})\frac{e^{i\theta}+z}{e^{i\theta}-z}\, d\theta\\
 &=i\beta+\int_{\mathbb T}g(rs)\frac{s+z}{s-z}\, dm(s).
 \end{align*}
Now, $\{g(rs)dm(s)\}_{0<r<1}$ is a family of positive measures on $\mathbb T$ all having the same mass $\alpha$. By weak-star compactness, there is a positive measure $\mu$ on $\mathbb T$ such that $\mu=\lim_{n}g(r_ns)dm(s)$ weak-star for some sequence $r_n$ tending to $1$. Hence,
$$
f(z)=\lim_{n\to \infty}f(r_nz)=i\beta+\lim_{n\to \infty}\int_{\mathbb T}g(r_ns)\frac{s+z}{s-z}\, dm(s)=i\beta+\int_{\mathbb T}\frac{s+z}{s-z}\, d\mu(s).
$$
Thus, we have found an integral representation of $f$.
\end{proof}

 \begin{prop}[Integral representation of Nevanlinna--Pick functions]
There is a one-to-one correspondence between the functions $f$ in $\mathcal N$ and the triples $(a,b,\tau)$ where $a\geq 0$, $b\in \mathbb R$ and $\tau \in M^+(\mathbb R)$ is  finite. It is  given by
$$
f(z)=az+b+\int_{-\infty}^{\infty}\frac{tz+1}{t-z}\, d\tau(t),\quad z\in \mathbb H.
$$
Writing $d\mu(t)=(t^2+1)d\tau(t)$ the representation takes the form
\begin{equation}
 \label{eq:nevanlinna-pick-integral}
f(z)=az+b+\int_{-\infty}^{\infty}\left(\frac{1}{t-z}-\frac{t}{t^2+1}\right)\, d\mu(t),\quad z\in \mathbb H.
\end{equation}
\end{prop}
\begin{proof}
 We use a conformal mapping $\varphi:\mathbb H\to B(0,1)$ to transform functions in the upper half-plane to functions in the unit disk and apply the integral representation of Herglotz functions.

 Let $\varphi(w)=\nicefrac{(w-i)}{(w+i)}$ and notice that $\varphi^{-1}(z)=\nicefrac{i(1+z)}{(1-z)}$. For $f\in \mathcal N$, $g(z)=-if(\varphi^{-1}(z))$ is a Herglotz function. It thus admits a representation
 $$
 -if(\varphi^{-1}(z))=ib+\int_{\partial B(0,1)}\frac{s+z}{s-z}\, d\mu(s).
 $$
 In this relation we put $w=\varphi^{-1}(z)$ thus obtaining
 $$
 -if(w)=ib+\int_{\partial B(0,1)}\frac{s+\varphi(w)}{s-\varphi(w)}\, d\mu(s).
 $$
 Next, let $\tau=  \varphi^{-1}(\mu)$ and notice that $\tau$ is a positive and finite measure on $\mathbb R\cup \{\infty\}$. This gives
 \begin{align*}
 f(w)&=-b+i\int_{\mathbb R\cup \{\infty\}}\frac{\varphi(t)+\varphi(w)}{\varphi(t)-\varphi(w)}\, d\tau(t)\\
 &=-b+\int_{\mathbb R\cup \{\infty\}}\frac{tw+1}{t-w}\, d\tau(t)\\
 &=-b+\tau(\{\infty\})w+\int_{-\infty}^{\infty}\frac{tw+1}{t-w}\, d\tau(t).
 \end{align*}
This shows that any Nevanlinna--Pick function admits an integral representation of the asserted form.
\end{proof}

\begin{rem}
Notice that $\int_{-\infty}^{\infty}\nicefrac{d\mu(t)}{(t^2+1)}<\infty$.
 Warning: It tempting to split the integral \eqref{eq:nevanlinna-pick-integral} into the sum of two,
 $$
 \int_{-\infty}^{\infty}\frac{d\mu(t)}{t-z}-\int_{-\infty}^{\infty}\frac{td\mu(t)}{t^2+1}.
 $$
 This is in general not allowed, since both integrals may diverge.
 \end{rem}




\begin{prop}
[Expressing the triple in terms of $f$] The triple $(a,b,\mu)$ in \eqref{eq:nevanlinna-pick-integral} is expressed as follows
$$
a=\lim_{y\to \infty}\nicefrac{f(iy)}{iy},\quad b=\Re f(i), \quad \mu=\lim_{y\to 0_+}\nicefrac{1}{\pi}\Im f(x+iy)\, dx\ \text{(vaguely)}.
$$
\end{prop}
\begin{prop}
[Identifying the limit measure]
 Let $f\in \mathcal N$ have the representation \eqref{eq:nevanlinna-pick-integral} and assume that
$$
s(x) = \nicefrac{1}{\pi}\lim_{y\to 0^+}\Im f(x+iy)
$$
exists for $x$ in an open set $I$ of $\mathbb R$ and that the convergence is uniform on compact
subsets of I. Then the restriction of $\mu$ to $I$ has  density $s$ with respect to
Lebesgue measure.
\end{prop}

\begin{example} Integral representation of two elementary Nevanlinna--Pick functions.
\begin{enumerate}
\item A complicated way of writing $i$:
$$
i=\int_{-\infty}^{\infty}\left(\frac{1}{t-z}-\frac{t}{t^2+1}\right)\, dt.
$$
\item The principal logarithm
$$
\log z=\int_{-\infty}^0\left(\frac{1}{t-z}-\frac{t}{t^2+1}\right)\, dt.
$$
\end{enumerate}

\end{example}


 {}
 \begin{exercise}
  Show that $\log$ is a Nevanlinna--Pick function and derive the formula
 $$
 \log z=\int_{-\infty}^0\left(\frac{1}{t-z}-\frac{t}{t^2+1}\right)\, dt.
 $$
 Is it permissible to integrate term by term?
 \end{exercise}

 \begin{exercise}
  Show that $f(z)=-\nicefrac{\log (1+z)}{z}$ is a Nevanlinna--Pick function and derive the formula
 $$
 -\frac{\log (1+z)}{z}=-\frac{\pi}{4}+\int_{-\infty}^{-1}\left(\frac{1}{t-z}-\frac{t}{t^2+1}\right)\frac{dt}{-t}
 $$
 for $z\in \mathbb C\setminus (-\infty, -1]$.
Is it permissible to integrate term by term? {\em [Hint: assume that $z=x>0$ and work out the integral on the right hand side. Alternatively, apply a maximum principle to obtain positivity of the imaginary part of the function on the left hand side.]}
 \end{exercise}


 \begin{exercise}
  Show that
  $$
  f(z)=\frac{\log \Gamma(z+1)}{z}
  $$
  is a Nevanlinna--Pick function and derive the formula
 $$
 f(z)=-\gamma+\sum_{k=1}^{\infty}\left(\nicefrac{1}{k}-\arctan \nicefrac{1}{k}\right)+
 \int_{-\infty}^{-1}\left(\frac{1}{t-z}-\frac{t}{t^2+1}\right)c(t)\, dt,
 $$
 where $c(t)=\nicefrac{(1-k)}{t},$ for $t\in (-k,1-k)$ and $k=2,3,\ldots$ Is it permissible to integrate term by term?
{\em [Hint: Use the sum representation \eqref{eq:logGamma} of $\log \Gamma(z+1)$ to obtain the formula
 $$
 \frac{\log \Gamma(z+1)}{z}=-\gamma +\sum_{k=1}^{\infty}\left(\frac{1}{k}-\frac{\log(1+z/k)}{z}\right).
 $$
 Thereafter, use the previous exercise.]}
\end{exercise}

\section{Second lecture: Crossing the real line. Stieltjes functions. Positive definite kernels}
In this lecture we continue the study of the Nevanlinna--Pick functions and the related class of Stieltjes functions together with generalized Stieltjes functions of positive order.

We show that the function \eqref{eq:logGammaratio} related to Euler's gamma function belongs to the Nevanlinna--Pick class and that this gives a solution to the problem used as a motivation in the previous lecture. References: \cite{BP-1} and \cite{BP-2}.

If time permits we shall briefly introduce the so-called L\"owner kernel and its relation to Nevanlinna--Pick functions. References: \cite{D}.
 \begin{exercise}
  [Recap]
  Which of the following functions belong to $\mathcal N$?
  \begin{enumerate}
  \item $f(z)=\sqrt{z+1}$
  \item $f(z)=z^2$
  \item $f(z)=e^{-z}$
 \end{enumerate}
 \end{exercise}

\subsection{The function $\log \Gamma(z+1)/(z\log z)$}
Let us return to the motivating example mentioned in the previous lecture.
 \begin{prop}
 The function $f(z)=\nicefrac{\log \Gamma(z+1)}{(z\log z)}$ is a Nevanlinna--Pick function and
$$
 \displaystyle{\frac{\log \Gamma(z+1)}{(z\log z)}=1-\int_0^{\infty}\frac{d(-t)}{t+z}\, dt},
 $$
 where the graph of $d$ is given in Figure \ref{fig:d}. In particular $(-1)^{n-1}f^{(n)}(x)>0$ for $n\geq 1$.
 \end{prop}
\begin{figure}
 \begin{center}
 \includegraphics[scale=0.25]{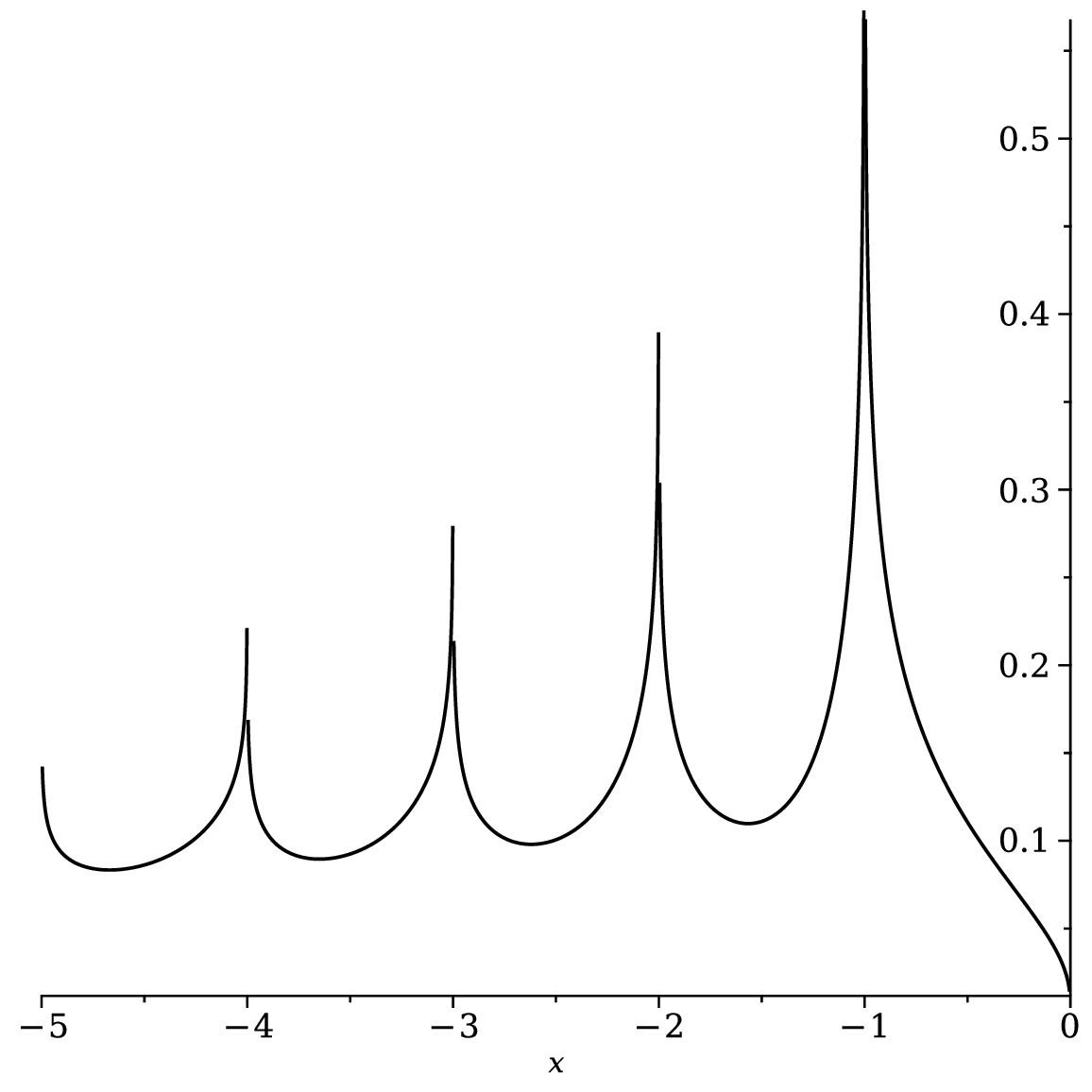}
\end{center}
\caption{The graph of $d$}
\label{fig:d}
\end{figure}

 {We shall not give a complete proof of this result but outline the main strategy in items (1)--(5) below.}
 \begin{enumerate}
  \item Extend $\log \Gamma$ to the complex plane cut along the non-positive real axis and find its boundary values. We have
\begin{align*}
 \log \Gamma(z+1)&=-\gamma z -\sum_{k=1}^{\infty}\left(\log(1+z/k)-z/k\right),\quad \text{so that}
 \\
  \lim_{z\to t, z\in \mathbb H}\log \Gamma(z)&=\log |\Gamma(t)|-ik\pi, \quad t\in [-k,1-k],\ k\geq 1.
  \end{align*}
  \item Consider a harmonic function and its boundary values.
  Define
 \begin{align*}
 V(z)&=\Im \left(\frac{\log \Gamma(z+1)}{z\log z}\right)\\
 &=\frac{1}{|z\log z|^2}\Big((x\log |z|-y\arg z)\arg \Gamma(z+1)\\
 &\phantom{=}\qquad\qquad -(y\log |z|+x\arg z)\log |\Gamma(z+1)|\Big),
 \end{align*}
 and show that
 $$
  \lim_{z\to t, z\in \mathbb H}V(z)=d(t),\ t\in\mathbb R,
  $$
  where $d(-k)=\infty$ for $k\in -\mathbb N$, $d(t)=0$ for $t\geq 0$ and
  $$
 d(t)=-\frac{\log |\Gamma(t+1)|+(k-1)\log |t|}{t((\log |t|)^2+\pi^2)},\quad t\in (-k,1-k),\ k\geq 1.
 $$
  \item Check positivity of the boundary values, $d$, of the harmonic function $V$. To see that
 $d(t)\geq 0$ for $t\in \mathbb R$ notice that
$$
-t((\log |t|)^2+\pi^2)d(t)=\log |\Gamma(t+1)|+(k-1)\log |t|,
 $$
 for $t\in (-k,1-k),\ k\geq 1$, and use the functional equation of the Gamma function.
  \item Control the growth of $V$ in the upper half-plane:
  $
  V(z)\geq C\ \text{for all}\  z\in \mathbb H \ \text{of sufficiently large absolute value.}
  $
  \item Find the integral representation.
 \end{enumerate}

\subsection{Extension across points of the real line}
Any $f\in \mathcal N$ is extended to $\mathbb C\setminus \mathbb R$ by reflection (meaning that we put $f(\overline{z})=\overline{f(z)}$ for $z\in \mathbb H$). Sometimes this extension is holomorphic across points of the real line.
\begin{example} Extending two elementary functions.
 \begin{itemize}
  \item $f(z)=\log z$ can be extended in this way to $\mathbb C\setminus (-\infty,0]$;
  \item $f(z)=i$ cannot be extended in this way across any non-empty open subset of $\mathbb R$.
 \end{itemize}

\end{example}

 \begin{prop}[Extension of Nevanlinna--Pick functions across the real line] Let $A$ be a closed subset of $\mathbb R$. The extension by reflection of a function $f\in \mathcal N$ represented in \eqref{eq:nevanlinna-pick-integral} by the triple $(a,b,\mu)$ can be extended holomorphically to $\mathbb C\setminus A$ if and only if $\supp \mu\subseteq A$.
 \end{prop}
 \begin{notation}
  The class of functions with the property in the proposition above is denoted by $\mathcal N_A$.
 \end{notation}


\begin{prop}[Extension across $(0,\infty)$]
\label{prop:ext}
 The following conditions for a function $f:(0,\infty)\to [0,\infty)$ are equivalent:
 \begin{enumerate}
  \item $f$ can be extended to a function in $\mathcal N_{(-\infty,0]}$;
  \item $f$ admits the representation
  $$
  f(x)=cx+x\int_0^{\infty}\frac{d\sigma(t)}{t+x},
  $$
  where $c\geq 0$ and $\sigma\in M^+([0,\infty))$ satisfies $\int_0^{\infty}\nicefrac{d\sigma(t)}{(t+1)}<\infty$.
  \end{enumerate}

\end{prop}
\begin{proof}
 (1) implies (2): We have the representation
 $$
 f(x)=ax+b+\int_0^{\infty}\frac{tx-1}{t+x}\, d\tau(t),
 $$
 where $\tau$ is renamed as its reflection in $0$. By differentiation it follows that $f'(x)\geq 0$ for $x>0$. Hence, $f$ being non-negative on $(0,\infty)$, $\lim_{x\downarrow 0}f(x)$ is finite. Therefore, by Lebesgue's monotone convergence theorem we get from the relation above
 $$
 b-\int_0^{\infty}\frac{d\tau(t)}{t}=\lim_{x\downarrow 0}f(x)=f(0+).
 $$
 Hence $\tau$ has no mass at the origin and $\int_0^{\infty}\frac{d\tau(t)}{t}<\infty$. The rest is computation:
 \begin{align*}
  f(x)&=ax+b-\int_0^{\infty}\frac{d\tau(t)}{t}+\int_0^{\infty}\left(\frac{1}{t}+\frac{tx-1}{t+x}\right)\, d\tau(t)\\
  &=ax+b-\int_0^{\infty}\frac{d\tau(t)}{t}+x\int_0^{\infty}\frac{1}{x+t}\left(t+\frac{1}{t}\right)\, d\tau(t)\\
  &=ax+x\int_0^{\infty}\frac{d\sigma(t)}{x+t},
 \end{align*}
 with
 $$
 \sigma=\left(b-\int_0^{\infty}\frac{d\tau(t)}{t}\right)\epsilon_0+(t+1/t)d\tau(t).
 $$
(Here $\epsilon_0$ denotes the point mass at $0$.)

(2) implies (1) is easy.
\end{proof}


\begin{example}
 We have
\begin{itemize}
 \item $z\mapsto \log z$ is in $\mathcal N_{(-\infty,0]}$ but is not positive on $(0,\infty)$,
 \item $z\mapsto \log (z+1)$ is in $\mathcal N_{(-\infty,0]}$ and is also  positive on $(0,\infty)$,
 \item $z\mapsto \nicefrac{\log \Gamma(z+1)}{z\log z}$ is in $\mathcal N_{(-\infty,0]}$ and is also  positive on $(0,\infty)$.
\end{itemize}

\end{example}

\subsection{Stieltjes functions}
 \begin{defn}
  [The class $\mathcal S$] A function $f:(0,\infty)\to \mathbb R$ is a Stieltjes function if
  $$
  f(x)=c+\int_0^{\infty}\frac{d\sigma(t)}{t+x},
  $$
  for some $c\geq 0$ and $\sigma\in M^+([0,\infty))$ making the integral converge. The class of Stieltjes functions is denoted by $\mathcal S$.
 \end{defn}

 The proof of the next result follows from Proposition \ref{prop:ext}.
 \begin{prop}{Theorem}
  $f\in \mathcal S$ if and only if $x\mapsto xf(x)$ can be extended to a function in $\mathcal N_{(-\infty,0]}$.
 \end{prop}
 Another characterization of Stieltjes functions is the following.
 \begin{prop}
 A function $f:(0,\infty)\to \mathbb R$ belongs to $\mathcal S$ if and only if
 \begin{enumerate}
  \item[(i)] $f(x)\geq 0$ for $x>0$, and
  \item[(ii)] $f$ has a holomorphic extension to $\mathbb C\setminus (-\infty,0]$ with $\Im f(z)\leq 0$ for $\Im z>0$.
 \end{enumerate}

\end{prop}
\begin{proof}
 If $f$ satisfies (i) and (ii) then $-f\in \mathcal N_{(-\infty,0]}$ and thus there exist $a\geq 0$, $b\in \mathbb R$ and a finite positive measure $\tau$ on $[0,\infty)$ such that
 $$
 -f(x)=ax+b+\int_0^{\infty}\frac{tx-1}{t+x}\, d\tau(t)=ax+b+x\int_0^{\infty}\frac{td\tau(t)}{t+x} -\int_0^{\infty}\frac{d\tau(t)}{t+x}.
 $$
 Now use $f(x)\geq 0$ to get
 $$
 ax+b\leq ax+b+x\int_0^{\infty}\frac{td\tau(t)}{t+x} \leq\int_0^{\infty}\frac{d\tau(t)}{t+x}\to 0,\quad \text{as} \ x\to \infty.
 $$
 This forces $a$ to be equal to $0$ and consequently,
 $$
 x\int_0^{\infty}\frac{td\tau(t)}{t+x} \leq-b+\int_0^{\infty}\frac{d\tau(t)}{t+x}.
 $$
 By monotone convergence (letting $x\to \infty$) we see that $\int_0^{\infty}t\, d\tau(t)\leq -b$. This gives
 \begin{align*}
 f(x)&=-b+\int_0^{\infty}\frac{1-tx}{t+x}\, d\tau(t)\\
 &=-b+\int_0^{\infty}\frac{-t(x+t)+t^2+1}{t+x}\, d\tau(t)\\
 &=-b-\int_0^{\infty}t\, d\tau(t)+\int_0^{\infty}\frac{t^2+1}{t+x}\, d\tau(t),
 \end{align*}
 showing that $f$ belongs to $\mathcal S$.
\end{proof}

%

 \begin{defn}
  [The class $\mathcal S_{\lambda}$] Let $\lambda>0$.
A function $f:(0,\infty)\to \mathbb R$ is a generalized Stieltjes function of positive order $\lambda$ if
  $$
  f(x)=c+\int_0^{\infty}\frac{d\mu(t)}{(t+x)^{\lambda}},
  $$
  for some $c\geq 0$ and a positive measure $\mu$ on $[0,\infty)$ for which
  $$
  \int_0^{\infty} \frac{ d\mu(t)}{(1+t)^{\lambda}}<\infty.
  $$
The class of these functions is denoted as $\mathcal S_{\lambda}$.
\end{defn}

Notice that  $\mathcal S_1=\mathcal S$ is the class of ordinary Stieltjes functions.


\begin{exercise}
 Show that
 \begin{enumerate}
  \item $f(x)=\nicefrac{1}{x^{\alpha}}\in \mathcal S$ if and only if $0\leq \alpha\leq 1$,
  \item $f(x)=\nicefrac{\log(1+x)}{x}\in \mathcal S$,
  \item $f(x)=\log\left(\nicefrac{(x+b)}{(x+a)}\right)\in \mathcal S$ for $0\leq a\leq b$, and in particular that $\log\left(1+\nicefrac{1}{x}\right)$ belongs to $\mathcal S$.
 \end{enumerate}

\end{exercise}
\begin{exercise}
 Suppose that $f\in \mathcal S\setminus \{0\}$. Show that $x\mapsto \nicefrac{1}{(xf(x))}\in \mathcal S$.
\end{exercise}

\begin{exercise}
 Let $f\in \mathcal N_{(-\infty,0]}$. Show that there exist $c\geq 0$ and a positive measure $\mu$ on $[0,\infty)$ such that
 $$
 f'(x)=c+\int_0^{\infty}\frac{d\mu(t)}{(t+x)^{2}}.
  $$
 Thus the derivative of $f$ is a generalized Stieltjes function of order $2$.
\end{exercise}


\subsection{Operator monotone functions (optional)}
  Let $I$ be an open interval of the real line. A continuous function $f:I\to \mathbb R$ is operator monotone if for any two self adjoint operators $\mathbf A$ and $\mathbf B$ in a Hilbert space with their spectrum contained in $I$ we have
  $$
  \mathbf A\leq \mathbf B\Rightarrow f(\mathbf A)\leq f(\mathbf B).
  $$
  Notice that when $\mathbf A$ is a $n\times n$ matrix, $f(\mathbf A)$ can be defined using diagonalization; in the general case $f(\mathbf A)$ is defined via the functional calculus. L\"{o}wner proved the following fundamental results.
  \begin{itemize}
   \item For  $f$ to be operator monotone it is enough to consider $n\times n$ matrices for any $n$.
   \item The class of operator monotone functions is exactly $\mathcal N_{\mathbb R\setminus I}$.
  \end{itemize}

%
%
%

\begin{defn}
 A function $K:I\times I\to \mathbb R$ is a positive definite kernel if for any $n$, any points $t_1,\ldots,t_n\in I$ and any complex numbers $z_1,\ldots,z_n$,
$$
\sum_{k=1}^nK(t_i,t_j)z_i\overline{z_j} \geq 0.
$$
\end{defn}
\begin{defn}For a $C^1$-function $f:I\to \mathbb R$ the {\em L\"owner kernel} $K_f$ is defined as
$$
K_f(s,t)=\nicefrac{(f(s)-f(t))}{(s-t)},\quad s\neq t, \qquad  K_f(t,t)=f'(t).
$$
\end{defn}

\begin{prop}
[L\"{o}wner's theorem]
Let $I$ be an open interval and let $f:I\to \mathbb R$. The following are equivalent:
\begin{enumerate}
\item $f$ is operator monotone on $I$.
\item The L\"{o}wner kernel $K_f$ is positive definite on $I\times I$.
\item $f$ has a holomorphic  extension to $\mathbb C\setminus (\mathbb R\setminus I)$ to a Nevanlinna--Pick function.
\end{enumerate}
\end{prop}

\section{Third lecture: Inverses  and higher order gamma functions}
Inverse functions to Euler's gamma function are investigated. References: \cite{P-inverse-gamma}.

The so-called $G$-function of Barnes is introduced (\cite{Barnes-double}) together with a hierarchy of higher order (double, triple, \ldots) gamma functions. See also \cite{Ruijsenaars}. We mention briefly a related Nevanlinna--Pick function (similar to the function from the previous lecture)  in  the framework of the $G$-function. Reference: \cite{Double-Pick}.

Asymptotic expansions of higher order gamma functions are described with focus on the remainders in these expansions.
References: \cite{Double}, \cite{Triple-I} and \cite{Triple-II}.

 \begin{exercise}
  {Recap -- Stieltjes functions. Are the following statements true or false?}
  \begin{enumerate}
  \item $\mathcal N _{(-\infty,0]}=\mathcal S$
  \item $f,g\in \mathcal S \Rightarrow fg\in \mathcal S$
 \end{enumerate}
 \end{exercise}

 \begin{exercise}
  {Discuss} ways of proving that a given function $f$ defined on the positive half-line belongs to $\mathcal S$.
 \end{exercise}

\subsection{Inverse functions}
We shall investigate mapping properties of inverses of the Gamma function using conformal mapping.
  {Recall that the mapping $\log \Gamma$} is given by
  \begin{align*}
   \log \Gamma(z)&=-\log z-\gamma z-\sum_{k=1}^{\infty}
   \left(\log \left(1+\nicefrac{z}{k}\right)-\nicefrac{z}{k}\right)
   \end{align*}
   for $z\in \mathbb C\setminus (-\infty,0]$.
  Its derivative is the so-called digamma function (or $\psi$-function):
$$
\psi(z)=\Gamma'(z)/\Gamma(z)=(\log \Gamma)'(z)=-\gamma+\sum_{k=0}^{\infty}\left(\nicefrac{1}{(k+1)}-\nicefrac{1}{(k+z)}\right).
$$
Figure \ref{fig:gamm-plot-negative} shows a plot of the Gamma function also on the negative line. This plot indicates the fact that there is exactly one extremal value for $\Gamma$ on each of the intervals $(-k,-k+1)$ for $k\geq 1$.
\begin{figure}
 \includegraphics[scale=0.4]{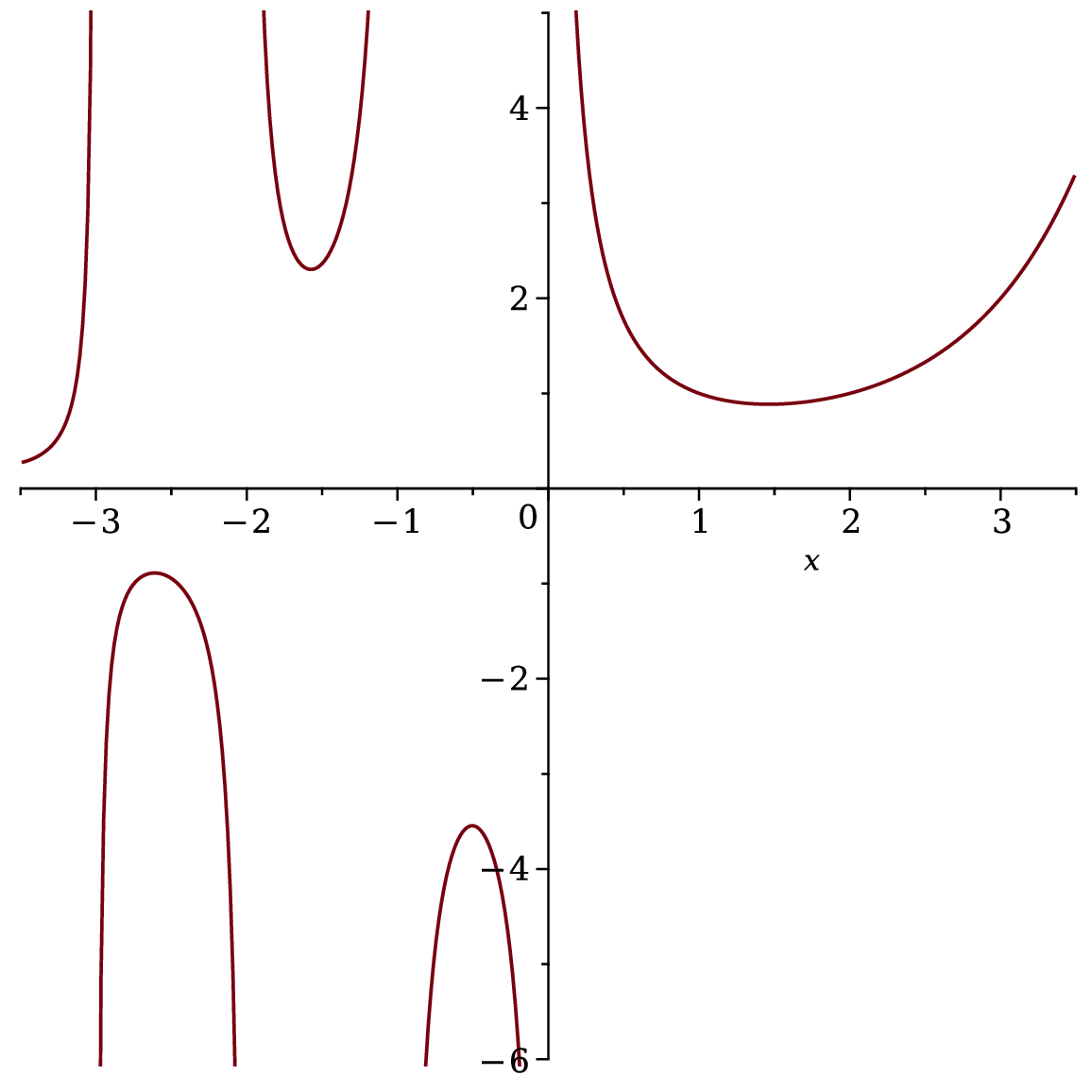}
 \caption{$\Gamma$ on the real line}
 \label{fig:gamm-plot-negative}
 \end{figure}

   \begin{defn}
    We let
    \begin{itemize}
     \item $x_k$ be the minimum point of $|\Gamma|$ on  $(-k,-k+1)$ for $k\geq 1$.
     \item $x_0$ be the minimum point of $\Gamma$ on the positive half-line.
    \end{itemize}

   \end{defn}


 \begin{prop}
  The mapping $\log \Gamma: \mathbb H\to \mathbb C$ is conformal and
  $$
  \log \Gamma(\mathbb H)=\mathcal V=\mathbb C\setminus \cup_{k=0}^{\infty} [\log|\Gamma(x_k)|,\infty)\times \{-ik\pi\}.
  $$
  The domain $\mathcal V$ is sketched in Figure \ref{fig:V}.
 \end{prop}
\begin{figure}
 \includegraphics[scale=0.4]{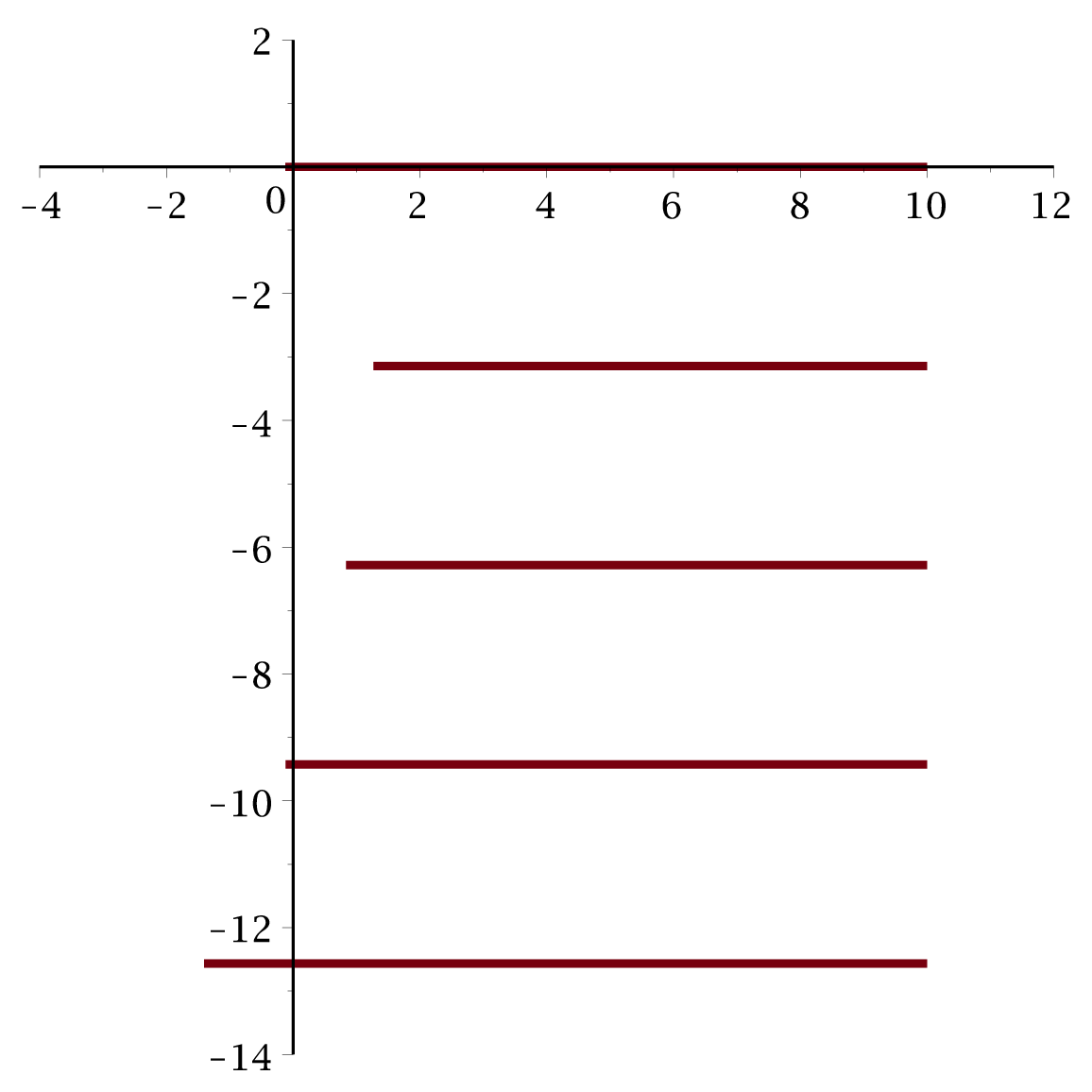}
 \caption{The domain $\log \Gamma (\mathbb H)$}
 \label{fig:V}
\end{figure}
The construction of inverses is now very natural: For $k\in \mathbb Z$ the function $z\mapsto \log z-i(k+1)\pi$ maps $\mathbb H$ to
$$
S_k=\{z\in \mathbb C| \Im z\in (-(k+1)\pi,-k\pi)\}\subseteq \mathcal V.
$$
Hence the function $g_k$ defined as
$$
g_k(z)=(\log \Gamma)^{-1}\left(\log z-i(k+1)\pi\right), \quad z\in \mathbb H
$$
maps $\mathbb H$ into $\mathbb H$ and is thus by definition a Nevanlinna--Pick function!
 \begin{prop}
The function $g_k$ is, for any $k\in \mathbb Z$, a Nevanlinna--Pick function and $$
\Gamma(g_k(z))=(-1)^{k+1}z
$$
for $z\in \mathbb H $.
\end{prop}
\begin{proof}
 $
\Gamma(g_k(z))=e^{\log \Gamma(g_k(z))}=e^{\log z-i(k+1)\pi}=(-1)^{k+1}z
$.
\end{proof}

\begin{prop}
The restriction of $\Gamma$ to $D_k=\{ z\in \mathbb H\, |\, \log \Gamma(z)\in S_k\}$ is conformal.
\end{prop}
\begin{proof}
 We notice that
$
D_k=\{ z\in \mathbb H\, |\, \arg \Gamma(z)\in (-(k+1)\pi,-k\pi)\}
$.
If $\Gamma(z_1)=\Gamma(z_2)$ for $ z_1,z_2\in D_k$ then $\log|\Gamma(z_1)|=\log|\Gamma(z_2)|$ and $\arg \Gamma(z_1)=\arg \Gamma(z_2)+2\pi l$ for some $l\in \mathbb Z$. Since $z_1,z_2\in D_k$, $|\arg \Gamma(z_1)-\arg \Gamma(z_2)|<\pi$ so that $l=0$ and  $\log \Gamma(z_1)=\log \Gamma(z_2)$. By conformality of $\log \Gamma$, $z_1=z_2$.
\end{proof}



%
%
%
\subsection{Barnes' $G$-function}
 Ernest William Barnes (1874--1953) was an English mathematician and scientist. Later he became a liberal theologian and bishop.
Barnes worked with various aspects of the gamma function. He introduced a hierarchy of so-called multiple (also called or multi- or higher order) gamma functions in the beginning of the twentieth century. There is also a more recent study by Ruijsenaars, in particular of asymptotic expansions, similar to Stirling's formula.

 We recall that the Gamma function satisfies the functional equation
 $\Gamma(z+1)=z\Gamma(z)$.
 Barnes considered a ``higher order'' functional equation
 \begin{equation}
  \label{eq:Barnes-functional}
  G(z+1)=\Gamma(z)G(z)
 \end{equation}
 and investigated solutions $G$ to this equation.

\begin{itemize}
 \item Existence: the following entire function $G$ is a solution to \eqref{eq:Barnes-functional}!
$$
G(z+1)=(2\pi)^{z/2}e^{-((1+\gamma)z^2+z)/2}\prod_{k=1}^{\infty}\left(
  1+\frac{z}{k}\right) ^ke^{-z+z^2/2k}.
$$
\item Uniqueness: The functional equation \eqref{eq:Barnes-functional}, $G(1)=1$ and $\partial_x^3(\log G)(x+1)\geq 0$ for $x>0$ determine $G$.
\end{itemize}
The uniqueness results is similar to the uniqueness of solutions to the functional equation of the Gamma function provided logarithmic convexity holds. (This is the famous Bohr-Mollerup theorem.)

 As we have seen the function $\nicefrac{\log \Gamma(1+z)}{z\log z}$ is a Nevanlinna--Pick function. Does something similar hold for $G$? The answer is given in the next result.

\begin{prop}
 We have the representation
 $$
 \frac{\log G(z+1)}{z^2\log z}=\frac{1}{2}-\int_0^{\infty}\frac{d(t)}{t+z}\, dt
 $$ for some positive function $d$ on the positive line.

\end{prop}
A similar but more complicated statement holds for a so-called  triple gamma function. See \cite{Das}.


%
%
%

\subsection{Multiple gamma functions}
Let us briefly go through the construction of the so-called multiple (or higher order) gamma functions, following the procedure given in Ruijsenaars paper. To motivate the procedure we first consider a classical relation between the Gamma function and the  Hurwitz zeta function.  

 Hurwitz' zeta function (which is the famous Riemann zeta function when $x=1$) is defined by
$$
\zeta(s,x)=\sum_{n=0}^{\infty}\frac{1}{(n+x)^s}=\frac{1}{\Gamma(s)}\int_0^{\infty}\frac{e^{-xt}t^{s-1}}{1-e^{-t}}\, dt.
$$
For $x>0$, one may prove that the function $s\mapsto \zeta(s,x)$ has a holomorphic extension to $\mathbb C\setminus \{1\}$ with a simple pole at $s=1$ and Lerch's theorem states that
$$
\partial_s\zeta(s,x)_{s=0}=\log \Gamma(x)-\log \sqrt{2\pi}.
$$


The multiple zeta
function $\zeta_N$ is defined as
\begin{align*}
\zeta_N(z,w)&=\frac{1}{\Gamma(s)}\int_0^{\infty}\frac{1}{(1-e^{-t})^N}e^{-wt}t^{z-1}\, dt\\
&=\sum_{m_1,\ldots,m_N=0}^{\infty}(w+m_1+\ldots+m_N)^{-z}
\end{align*}
for $\Re z>N$ and $\Re w>0$.
It turns out that the function $z\mapsto \zeta_N(z,w)$ has a meromorphic extension to $\mathbb C$ with
simple poles only at $z=1,\ldots,N$. The multiple gamma function $\Gamma_N$ is then defined in terms of the function
$$
\log \Gamma_N(w)= \partial _z\zeta_N(z,w)|_{z=0}.
$$
In this way, $\Gamma_1(w)=\Gamma(w)/\sqrt{2\pi}$.
and $\Gamma_2(w)=(2\pi)^{w/2}/G(w)$.



 We shall focus on an investigation of the remainders in asymptotic expansions of the functions $\log \Gamma_N$. In order to do so we introduce the multiple Bernoulli polynomials as follows.
\begin{defn}
The multiple Bernoulli polynomials $B_{N,k}(x)$ are defined by
$$
\frac{t^Ne^{xt}}{(e^{t}-1)^N}=
\sum_{k=0}^{\infty}B_{N,k}(x)\frac{t^k}{k!}, \quad |t|<2\pi
$$
and the multiple Bernoulli numbers by
$B_{N,k}=B_{N,k}(0)$.
\end{defn}
These polynomials appear in an asymptotic expansion of $\log \Gamma_N$ due to Ruijsenaars, which we state in the next proposition.
\begin{prop}
 [Asymptotic expansion of $\log \Gamma_N$]
 For $\Re w>0$,
\begin{align*}
\log \Gamma_N(w)&= \frac{(-1)^{N+1}B_{N,N}(w)}{N!}\log w +(-1)^N\sum_{k=0}^{N-1}\frac{B_{N,k}w^{N-k}}{k!(N-k)!}\sum_{l=1}^{N-k}l^{-1}\\
&+\sum_{k=N+1}^m\frac{(-1)^kB_{N,k}w^{N-k}}{k!}(k-N-1)!+
 R_{N,m}(w),\quad m\geq N,
\end{align*}
where
$$
 R_{N,m}(w)= \int_{0}^{\infty}\frac{e^{-wt}}{t^{N+1}} \underbrace{\left(\frac{t^N}{(1-e^{-t})^N}-
    \sum_{k=0}^{m}\frac{(-1)^k}{k!}B_{N,k}t^k\right)} \, dt.
    $$
\end{prop}


  Our aim is to investigate the behaviour of the remainders $R_{N,m}$. It seems to be a complicated problem and we could only do it for $N=2$ and $N=3$. In order to do this we shall consider the underbraced expression in some detail. It is simply equal to $f(t)-T_m(t)$ with $f(t)=(t/(1-e^{-t}))^N$ and $T_m$ being the $m$'th order Taylor polynomial of $f$ centered at $t=0$. Does this expression change sign on the positive real line? The standard Lagrange type remainder involving  $f^{(m+1)}(\xi)$ for a suitable $\xi>0$ seems inconclusive and another method is needed.

Let us start by considering the case $N=1$. In this situation the expansion is the classical Binet formula:
\begin{align*}
 \log \Gamma(z+1)&=\log \sqrt{2\pi}+(z+\nicefrac{1}{2})\log z-z+\sum_{k=2}^{m}\frac{(-1)^kB_{k}}{k(k-1)}\frac{1}{z^{k-1}}\\
 &+\int_0^{\infty}\frac{e^{-zt}}{t^2}\left(\frac{t}{1-e^{-t}}-
    \sum_{k=0}^{m}\frac{(-1)^k}{k!}B_{k}t^k\right)\, dt.
 \end{align*}
%
The next exercise furnishes an alternative expression for the integrand in the right hand side of this formula, implying positivity.
 \begin{exercise}
  Consider the function $f(z)=z/(1-e^{-z})$ and its Taylor polynomial $T_n$ of order $n$ centered at $z=0$. Let $w\in \mathbb C\setminus 2\pi i \mathbb Z$.
  \begin{enumerate}
   \item Find all poles and the corresponding residues of the function
   $g(z)=f(z)/((z-w)z^{n+1})$.
   \item Let $\mathcal F_N$ denote the boundary of $\{z\, |\, |\Re z|\leq N,\, |\Im z|\leq 2\pi (N+\nicefrac{1}{2})\}$
   traversed once in the counter clockwise direction.
   Show that
   $$
   \int_{\mathcal F_N}g(z)\,dz \to 0\quad \text{for}\ N\to \infty.
   $$
   \item Show that
   $$
   f(w)-T_n(w)=w^{n+2}\sum_{p=1}^{\infty}\frac{1}{(2\pi p)^n}\frac{1}{(2\pi p)^2+w^2}
   $$
   for $w>0$.

  \end{enumerate}

 \end{exercise}

Consider next the case $N=2$. The remainder takes the form
  $$
R_{2,m}(w)= \int_{0}^{\infty}\frac{e^{-wt}}{t^3}\left( \frac{t^2}{(1-e^{-t})^2}-
    \sum_{k=0}^{m}\frac{(-1)^k}{k!}B_{2,k}t^k\right) \, dt.
$$
When $m$ is odd the integrand in general changes sign along the positive line, but when $m$ even there is more regularity: the integrand does not change sign.
This follows from the result below.

\begin{prop}
 For $m\geq 1$ we have
$$
(-1)^{m-1}R_{2,2m}(w)=\int_{0}^{\infty}e^{-wt}t^{2m-2}\nu_m(t) \,
dt,\quad \Re w>0,
$$
where
\begin{align*}
\nu_m(t)&=
\sum_{k=1}^{\infty}(2\pi k)^{1-2m}\left( \frac{4\pi k}{t^2+(2\pi
    k)^2}+\frac{8\pi kt}{(t^2+(2\pi k)^2)^2}+\right.\\
    &\phantom{=} \left.\frac{(2m-1)}{2\pi k}\frac{2t}{t^2+(2\pi
    k)^2}\right).
\end{align*}
\end{prop}
The proof of this alternative form of the integrand follows the same  lines as the exercise above.

By differentiation, denoted $\partial_x$, we observe the regular sign behaviour of the remainders themselves.
\begin{prop}
 For $m\geq 1$ and $k\geq 0$,
 $$
 (-1)^k\partial_x^k\, (-1)^{m-1}R_{2,2m}(x)>0\quad \text{for} \ x>0.
 $$
\end{prop}
In the case of $N=3$ there is a similar, but more complicated,  result for the even remainders.


\begin{prop}
 For $m\geq 6$, the integrand in
 $$
(-1)^mR_{3,2m}(x)= \int_{0}^{\infty}\frac{e^{-xt}}{t^4} (-1)^m\left(\frac{t^3}{(1-e^{-t})^3}-
    \sum_{k=0}^{2m}\frac{(-1)^k}{k!}B_{3,k}t^k\right) \, dt
$$
is positive.
\end{prop}
As before we immediately obtain the following sign behaviour of the remainders.
\begin{prop}
 For $m\geq 6$ and $k\geq 0$,
 $$
 (-1)^k\partial_x^k\, (-1)^{m}R_{3,2m}(x)>0\quad \text{for} \ x>0.
 $$
\end{prop}
 In the next lecture we shall investigate a new class of functions with derivatives alternating in sign.


\section{Fourth lecture: Complete monotonicity}
This lecture focuses on the class of completely monotonic functions and its relation with the class of Stieltjes functions.

The completely monotonic functions were introduced a little less than 100 years ago. A smooth function on the positive half-line is  completely monotonic if the $n$'th derivative is positive when $n$ is even, and otherwise negative.  Bernstein's theorem characterizes the completely monotonic functions in terms of the Laplace transform. References: \cite{S}.

The class of completely monotonic functions of positive order is introduced and as examples we mention remainders in asymptotic expansions of gamma functions. Also the class of logarithmically completely monotonic functions is investigated and its relation to generalized Stieltjes functions of positive order is described. We introduce the class of Horn-Bernstein functions in relation with a study  of
the elementary function $x\mapsto (1+\nicefrac{1}{x})^{ax}$.
References: \cite{BP-horn}.

 \begin{exercise}
  {Recap. Which of the next two statements are true? Which are false?}
  \begin{enumerate}
   \item The function
   $$
   \frac{t}{1-e^{-t}}-
    \sum_{k=0}^{m}\frac{(-1)^k}{k!}B_{k}t^k
    $$
    is positive on $(0,\infty)$.
    \item The inverse of Euler's Gamma function is a Nevanlinna--Pick function.
  \end{enumerate}

 \end{exercise}

\subsection{Completely monotonic functions}
\begin{defn}
 A $C^{\infty}$-function $f:(0,\infty)\to \mathbb R$ is completely monotonic (CM) if
 $$(-1)^nf^{(n)}(x)\geq 0\quad \text{for}\  n\geq 0 \ \text{and}\  x>0.
 $$
\end{defn}

\begin{example}
 The following functions are CM.
 \begin{enumerate}
  \item $f(x)=\nicefrac{1}{x}$,
  \item
  $
  f(x)=c+\int_0^{\infty}\frac{d\mu(t)}{(t+x)^{\lambda}}
  $ where $c,\mu\geq 0$, so $\mathcal S_{\lambda}\subseteq \text{CM}$,
  \item $f(x)=e^{-x}$.
 \end{enumerate}

\end{example}
\begin{exercise}

 Discuss how to prove that $f(x)=e^{-x}$ is not a generalized Stieltjes function.
\end{exercise}
\begin{exercise}
 Show that the product of two CM functions is again CM.
\end{exercise}

{Bernstein's theorem characterizes the completely monotonic function via the Laplace transform.}
\begin{defn}[Laplace transform]
The Laplace transform of a positive measure $\mu$ on $[0,\,\infty)$ is defined by
$$
\mathcal L(\mu)(x)=\int _{0}^{\infty} e^{-xt}\,d\mu(t).
$$
\end{defn}
\begin{prop}
 [Bernstein's theorem, Acta.\ Math.\ 1929]
A function $f$ is CM if and only if
$
f=
\mathcal{L}(\mu)
$
for a positive measure $\mu$ on $[0,\,\infty)$ and $\mu$ is uniquely determined by $f$.
\end{prop}
\begin{proof}
Let $f\in \text{CM}$ with $\lim_{x\to 0+}f(x)=1$ and
 $\lim_{x\to \infty}f(x)=0$. We aim at constructing a positive measure $\mu$ such that $f=\mathcal L(\mu)$. The Taylor expansion centered at $a>0$ with integral type remainder reads as follows
 \begin{align*}
  f(x)&=\sum_{k=0}^{n-1}\frac{(-1)^kf^{(k)}(a)}{k!}(a-x)^k+\int_x^a \frac{(-1)^nf^{(n)}(s)}{(n-1)!}(s-x)^{n-1}\, ds.
 \end{align*}
 Now comes the key to the proof: When $0<x<a$ all terms in the sum and the integrand are positive by the complete monotonicity! By monotone convergence we thus obtain
 $$
 \int_x^{\infty} \frac{(-1)^nf^{(n)}(s)}{(n-1)!}(s-x)^{n-1}\, ds=
 \lim_{a\to \infty}\int_x^a \frac{(-1)^nf^{(n)}(s)}{(n-1)!}(s-x)^{n-1}\, ds
 \leq f(x).
 $$
Consequently,
$$
\lim_{a\to \infty}\sum_{k=0}^{n-1}\frac{(-1)^kf^{(k)}(a)}{k!}(a-x)^k
=f(x)-\int_x^{\infty} \frac{(-1)^nf^{(n)}(s)}{(n-1)!}(s-x)^{n-1}\, ds.
$$
Putting
$$
c_n(x)=\lim_{a\to \infty}\sum_{k=0}^{n-1}\frac{(-1)^kf^{(k)}(a)}{k!}(a-x)^k,
$$
it follows that $c_n=c_n(x)$ does not depend on $x$ (use $\nicefrac{(a-x_1)}{(a-x_2)}\to 1$ as $a\to \infty$), and furthermore,
$$
c_n\leq c_n+\int_x^{\infty} \frac{(-1)^nf^{(n)}(s)}{(n-1)!}(s-x)^{n-1}\, ds=f(x)\to 0
$$
for $x\to \infty$. Hence $c_n=0$. This and a clever change of variable ($t=n/s$) give us
\begin{align*}
f(x)
&=\int_x^{\infty} \frac{(-1)^nf^{(n)}(s)}{(n-1)!}(s-x)^{n-1}\, ds\\
&=\int_x^{\infty} \left(1-\frac{x}{s}\right)^{n-1}\frac{(-1)^nf^{(n)}(s)}{(n-1)!}s^{n-1}\, ds\\
&=\int_0^{\infty} \left(1-\frac{x}{s}\right)_+^{n-1}\frac{(-1)^nf^{(n)}(s)}{(n-1)!}s^{n-1}\, ds\\
&=\int_0^{\infty} \left(1-\frac{tx}{n}\right)_+^{n-1}\frac{(-1)^nf^{(n)}(n/t)}{n!}(n/t)^{n+1}\, dt.
\end{align*}
We let
$$
f_n(t)=\frac{(-1)^nf^{(n)}(n/t)}{n!}(n/t)^{n+1}
$$
and letting $x$ tend to zero in the relation above it follows that $\int_0^{\infty}f_n(t)dt=1$ so that $\mu_n=f_n(t)dt$ is a probability measure. By vague compactness there is a positive measure $\mu$ and a subsequence $\{\mu_{n_k}\}$ such that $\lim_{k}\mu_{n_k}=\mu$ vaguely.
Next two facts are used:
\begin{enumerate}
 \item [Fact 1] $\left(1-\frac{tx}{n}\right)_+^{n-1}\to e^{-xt}$ uniformly for $t\geq 0$. This is elementary but a little tricky to verify.
 \item [Fact 2] $\lim_{k}\int_0^{\infty}e^{-xt}d\mu_{n_k}(t)=\int_0^{\infty}e^{-xt}d\mu(t)$. To see this use the vague convergence, that $\mu_{n_k}$ is a probability and that $e^{-xt}=o(1)$.
\end{enumerate}
These facts imply
\begin{align*}
f(x)
&=\int_0^{\infty} \left(1-\frac{tx}{n_k}\right)_+^{n-1}f_{n_k}(t)\, dt\\
&=\int_0^{\infty} \left(\left(1-\frac{tx}{n_k}\right)_+^{n-1}-e^{-xt}\right)f_{n_k}(t)\, dt\\
&\phantom{=}\ + \int_0^{\infty} e^{-xt}f_{n_k}(t)\, dt-\int_0^{\infty} e^{-xt}\, d\mu(t)+\int_0^{\infty} e^{-xt}\, d\mu(t)\\
&\to \int_0^{\infty} e^{-xt}\, d\mu(t)\quad \text{as}\ k\to \infty.
\end{align*}
The assumptions on the limit behaviour at $0$ and $\infty$ of $f$ can be removed but we shall not describe this part of the proof and the uniqueness is also skipped.
\end{proof}

We gather some facts about completely monotonic functions in the next proposition.
\begin{prop}
 Let $\{f_n\}$ is a sequence from CM. Then:
 \begin{itemize}
  \item If $f_n(x)\to f(x)$ for all $x>0$ (pointwise) then $f$ is CM and $f_n^{(k)}\to f^{(k)}$ uniformly on any $[\epsilon,\infty)$ for all $k\geq 0$.
  \item If $\sup_n \lim_{x\to 0}f_n(x)<\infty$ there is a subsequence $\{f_{n_k}\}$ and $f\in \text{CM}$ such that $f_{n_k} \to f$ pointwise.
 \end{itemize}

\end{prop}

Let us give a few remarks about the history of completely monotonic and positive definite functions.
The Fourier transform of a finite positive measure is positive definite:
If $f(x)=\int_{-\infty}^{\infty}e^{ixt}\, d\mu(t)$ then
$$
\sum_{i,j=1}^nf(x_i-x_j)c_i\overline{c_j}=\int_{-\infty}^{\infty}\left|\sum_{k=1}^nc_ke^{ix_kt}\right|^2\, d\mu(t)\geq 0
$$
for all $n\geq 1$, and all $x_1,\ldots,x_n\in \mathbb R$ and all $c_1,\ldots,c_n\in \mathbb C$.

A converse result was obtained by Bochner (Bochner's theorem, 1932-33) and states that if $f:\mathbb R\to \mathbb C$ is positive definite (meaning that the double sum in the relation above is non-negative) and continuous then $f$ is the Fourier transform of a positive finite measure on $\mathbb R$.

Not long after, Bochner's theorem was extended to a semigroup setting.
The non-negative integers with addition is a semigroup and Bochner's theorem in this setting gives a proof of Hamburger's theorem characterizing moment sequences.

Also the half-line $(0,\infty)$ with multiplication is a semigroup and in this case Bochner's theorem can be stated as follows. A function $f:(0,\infty)\to \mathbb R$ is positive definite, bounded and continuous if and only if $f$ is the Laplace transform of a finite positive measure on $[0,\infty)$. Notice that Bernstein's theorem was already obtained a few years earlier.


 \begin{exercise}
  Find $\mu$, $\sigma$ and $\tau$ such that $\mathcal L(\mu)(x)=\nicefrac{1}{(1+x)}$, $\mathcal L(\sigma)(x)=e^{-x}$ and $\mathcal L(\tau)(x)=\nicefrac{1}{(1-e^{-x})}$. {\em [Hint: for $\tau$  consider the geometric series!]}
 \end{exercise}

\begin{exercise}
 Show that
 $$
 \log \left(1+\nicefrac{1}{x}\right)=
 \int_0^{\infty}e^{-xt}\left(\frac{1-e^{-t}}{t}\right)\, dt.
 $$
\end{exercise}

\begin{exercise}
 Show that $f\in \mathcal S_{\lambda}$ if and only if there are $c\geq 0$ and $\varphi\in \text{CM}$ such that
 $$
 f(x)=c+\int_0^{\infty}e^{-xt}t^{\lambda-1}\varphi(t)\, dt.
 $$
 Use this relation to show that $\mathcal S_{\lambda_1}\subseteq\mathcal S_{\lambda_2}$ when $\lambda_1<\lambda_2$.
\end{exercise}



We saw in one of the exercises above that any Stieltjes function is completely monotonic. There is in fact a close relation between complete monotonicity and the Stieltjes class; this is the famous theorem of Widder.
 \begin{prop}
  [Widder's theorem]
  $$
  f\in \mathcal S_1 \Leftrightarrow (x^kf(x))^{(k)}\ \text{is CM for all}\  k\geq 0.
  $$
 \end{prop}
There are similar results concerning generalized Stieltjes functions of positive order. This was investigated in a paper by Sokal in 2010 \cite{Sokal}, and also by Koumandos and myself.
 \begin{prop}
  $$
  f\in \mathcal S_{\lambda} \Leftrightarrow
  c_k(f)(x)\equiv x^{1-\lambda}(x^{\lambda-1+k}f(x))^{(k)}\  \text{is CM for all}\ k\geq 0.
  $$
 \end{prop}
 Of course not all completely monotonic functions are generalized Stieltjes functions of some positive order but any completely monotonic function can be approximated by such functions.
 \begin{prop}
  Any completely monotonic function is the pointwise limit of a sequence of generalized Stieltjes functions of positive orders.

 \end{prop}


\subsection{Subclasses of completely monotonic functions}
\begin{defn}[The class $\text{CM}(\alpha)$]
  Let $\alpha\in \mathbb R$. A function $\varphi:(0,\infty)\to \mathbb R$ is completely monotonic of order $\alpha$ if $t^{\alpha}\varphi(t)$ is CM. This class of functions is denoted by CM$(\alpha)$.
\end{defn}
The class $CM(\alpha)$ can, for $\alpha>0$ be characterized in terms of extra properties of the representing measure in Bernstein's theorem. We give the version for $\alpha=1$.
\begin{prop}
A function $f$ is CM(1) if and only if
 $f(x)=\mathcal L(\xi)(x)$ where $\xi$ is of the form
 $\xi(t)=\mu([0,t])$, and where $\mu$ is a positive measure on $[0,\infty)$ such that $\mathcal L(\mu)(x)<\infty$ for all $x>0$.
\end{prop}
\begin{example}
The remainders in the asymptotic expansions of the higher order Gamma functions investigated in the previous lecture are examples of $CM(\alpha)$-functions:
$(-1)^{m-1}R_{2,2m}$ is in CM($m-1$) and
$(-1)^mR_{3,2m}$ is in CM($m-1$). See \cite{Double,Triple-I, Triple-II}.

\end{example}


\begin{defn}
 A function $f:(0,\infty)\to \mathbb R$ is logarithmically completely monotonic (LCM) if $-f'/f=-(\log f)'$ is CM.
\end{defn}
The class of logarithmically completely monotonic functions was introduced and characterized by Horn.

\begin{prop}[Horn]
 The following conditions for a $C^\infty$-function $f: (0,\infty)\to (0,\infty)$ are equivalent:
 \begin{enumerate}
  \item[(i)] $-(\log f)'=-f'/f$ is completely monotonic,
  \item[(ii)] $f^{c}$ is completely monotonic for all $c>0$,
  \item[(iii)] $f^{1/n}$ is completely monotonic for all $n=1,2,\ldots$.
 \end{enumerate}
\end{prop}
\begin{proof}
 We notice first that $\text{LCM}\subseteq \text{CM}$: Assume that $f\in \text{LCM}$ and show by induction that $(-1)^nf^{(n)}\geq 0$. The case $n=0$ is OK by assumption; and the case $n=1$ (although it is not really needed) follows from the relation $-f'=gf$, with $g=-f'/f$. The induction step is based on
 $$
 (-1)^{n+1}f^{(n+1)}=(-1)^n(gf)^{(n)}=\sum_{k=0}^n\binom{n}{k}(-1)^kg^{(k)}(-1)^{n-k}f^{(n-k)}.
 $$
 (i) implies (ii): If $f\in \text{LCM}$ then $-(\log (f^c))'=-c(\log f)'\in \text{CM}$ so that $f^c\in\text{LCM}\subseteq \text{CM}$.

 (ii) implies (iii): This is immediate.

 (iii) implies (i): If $f^{1/n}$ is CM then so is $-n(f^{1/n})'=-f'f^{-1+1/n}=gf^{1/n}$. Letting $n$ tend to infinity it follows that
 $$
 -f'/f=g=\lim_n gf^{1/n}\in \text{CM},
 $$
meaning that $f\in \text{LCM}$.
\end{proof}

Any completely monotonic function $f$ has a natural extension to a holomorphic function in the right half-plane
$\{z\, |\, \Re z>0\}$ via the formula
$$
f(z)=\int_0^{\infty}e^{-zt}\, d\mu(t).
$$
A completely monotonic function extended in this way may of course have zeros in the right half-plane, as the following example shows: Consider, for given $\lambda>2$, $f_{\lambda}(z)=z^{-\lambda}+1$. (We use the principal logarithm in defining the power.) Then, clearly, $f_{\lambda}(z_{\lambda})=0$ where $z_{\lambda}=e^{i\pi/\lambda}$ belongs to the open right half-plane.

The function $f_{\lambda}$ is of course CM, but it is not LCM, as the next proposition shows.

\begin{prop}
If $f$ is LCM then $f$ has no zeros in $\{z\, |\, \Re z>0\}$.
\end{prop}
The proof of this result follows from Horn's theorem, implying that for any given $n$, there is $g\in \text{CM}$ such that $f=g^n$. Thus a zero of $f$ in the open right half-plane would have multiplicity at least $n$.

As we saw in one of the exercises above, any generalized Stieltjes function of positive order is completely monotonic. A more delicate question is the following: Which classes $\mathcal S_{\lambda}$ are contained in the class of logarithmically completely monotonic functions?

In view of the example above, if $\mathcal S_{\lambda}\subseteq \text{LCM}$ then $\lambda\leq 2$. Furthermore, when $\lambda\leq 2$ the example does not give us a function with a zero in the open right half-plane. There is a good reason for this: It is a deep result that all generalized Stieltjes functions of order $2$ are LCM. This was obtained by Kristiansen, who solved a long standing conjecture. For references, see \cite{BKP-nielsen}.

\begin{prop}[Kristiansen's theorem]
 If $f\in \mathcal S_{2}$ then $f$ is logarithmically completely monotonic.
\end{prop}
%

Relations between generalized Stieltjes, logarithmically completely monotonic and completely monotonic functions  are illustrated in Figure \ref{fig:relations}.
\begin{figure}[htb]
 \begin{center}
  \includegraphics[scale=0.25]{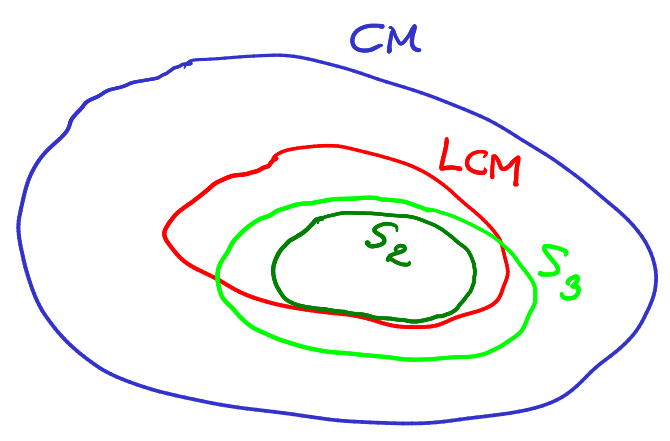}
 \end{center}
 \caption{CM, LCM and $\mathcal S_{\lambda}$}
 \label{fig:relations}
\end{figure}


%

\subsection{An elementary function}

In this section we investigate an elementary function and its relation to subclasses of completely monotonic functions.
 We define
 $$
 h_a(x)=(1+1/x)^{ax},\ \text{for}\  a>0,
 $$
 and aim at understanding monotonicity properties of $h_a$.

\begin{figure}
 \begin{center}
 \includegraphics[scale=0.25]{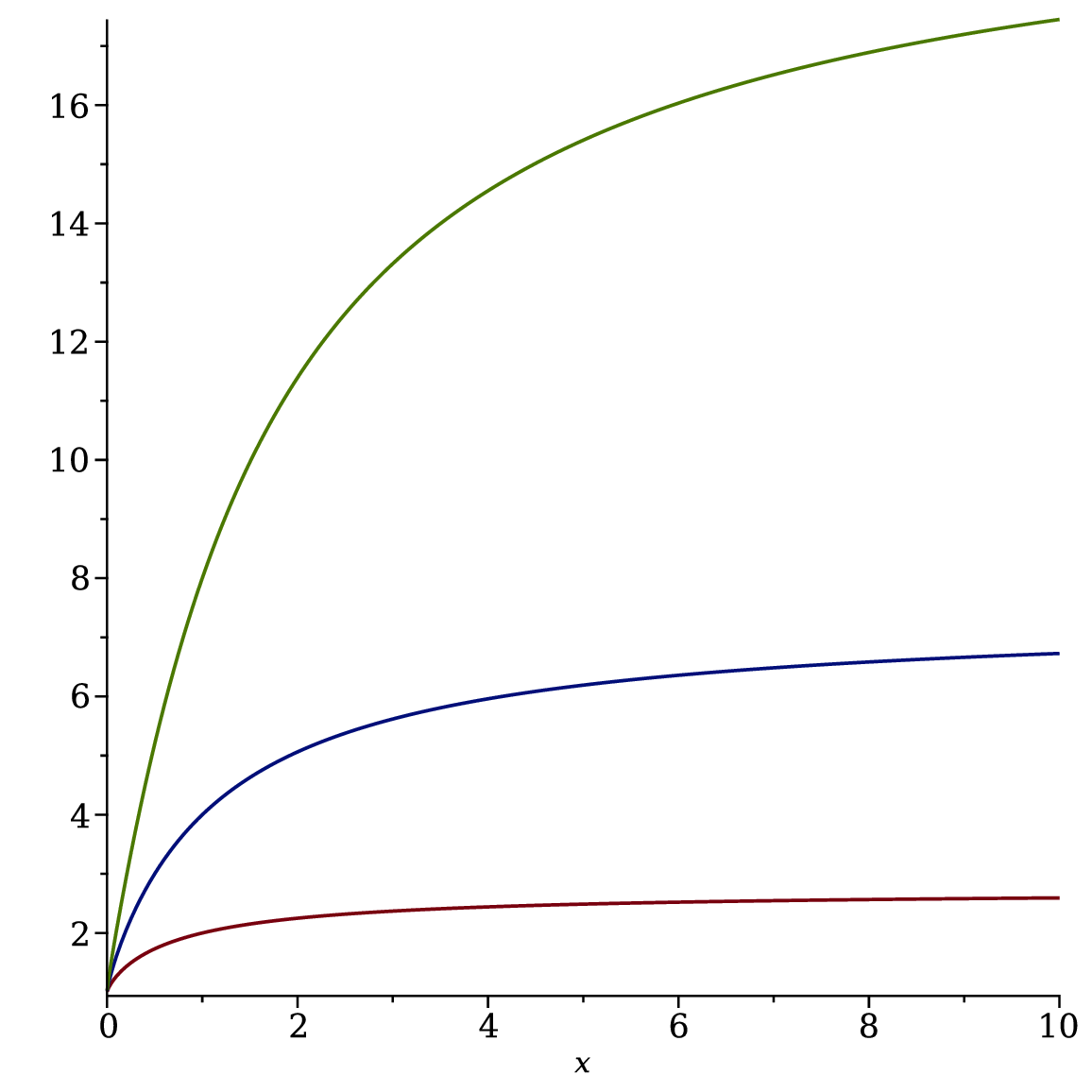}\ \includegraphics[scale=0.25]{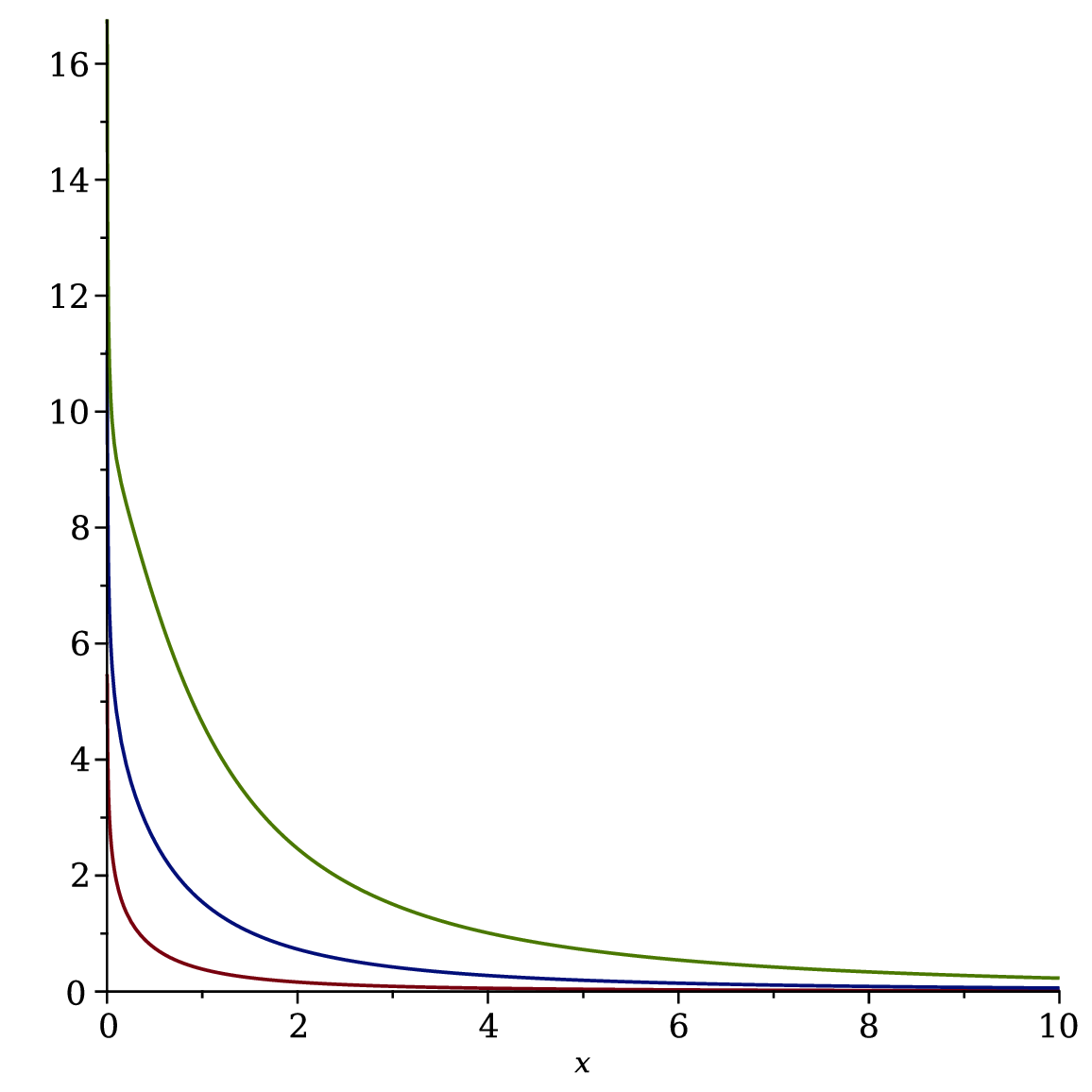}
\end{center}
\caption{The function $h_a$ and its derivative for $a=1,2,3$}
\label{fig:h}
\end{figure}

 In Figure \ref{fig:h} the graphs to the left show plots of $h_a$ and the graphs to the right plots of $h_a'$ for $a=1,2,3$. It is clear that all these functions are increasing and concave. However, looking closely on the plot of the derivatives, $h_3$ (in green colour) seems to be slightly irregular near to the origin. We shall see that differences in monotonicity properties appear when $a$ increases from $2$ to $3$.

 \begin{defn}
 A $C^{\infty}$-function $f:(0,\infty)\to (0,\infty)$ is a Horn-Bernstein function if $f'$ is LCM.
 \end{defn}

  \begin{center}
  {\bf Question}:
  For which $a$ do we have $h_a'\in $ LCM?
  \end{center}

\bigskip

   {We remark that  $h_a'\in \mathcal S_1$ if and only if $a\leq 1$} and that $h_a'\in $ CM if and only if $a\leq a^*\approx 2.299$. This was proved recently, see \cite{berg-massa-peron}.

\begin{prop}
 There is $b^*>0$ such that $h_a$ is a Horn-Bernstein function if and only if $a\leq b^*$. (And $b^*\approx 2.189$.)
 \end{prop}

 We shall outline the proof of this result, based on a certain Stieltjes function and a moment sequence. But we begin with some computation! The functions $\rho$ and $g$ are central in the investigation -- hence the red color.

 \begin{align*}
 (\log h_a)'(x)&=(ax\log(1+1/x))'=a\rho(x)\\
 {\color{red}\rho(x)}&=\log(1+1/x)-1/(1+x)\\
 h_a'(x)&=ah_a(x)\rho(x)
 \end{align*}

 \begin{align*}
 h_a''(x)&=ah_a'(x)\rho(x)+ah_a(x)\rho'(x)\\
 &=h_a'(x)(a\rho(x)+\rho'(x)ah_a(x)/h_a'(x))\\
 &=h_a'(x)(a\rho(x)+\rho'(x)/\rho(x))\\
 {\color{red}g(x)}&=-\rho'(x)/\rho(x)\\
 -h_a''(x)/h_a'(x)&=g(x)-a\rho(x)
 \end{align*}

 The key idea is now based on the observation
$$
h_a'\in \text{LCM} \ \Leftrightarrow \ -h_a''/h_a' \in \text{CM} \ \Leftrightarrow \  g-a\rho\in \text{CM}.
$$
It is easily seen that $\rho
\in \mathcal S_2\setminus \mathcal S_1$, and that
 \begin{align*}
   \rho(x)&=\log(1+1/x)-1/(1+x) =\mathcal L\left(\frac{1-e^{-t}}{t}-e^{-t}\right)(x).
 \end{align*}
The next result brings in the relation to Stieltjes functions.
\begin{prop} The function $g$ is a Stieltjes function and admits the following integral representation
   $$
   g(x)=\frac{1}{1+x}+\int_0^1\frac{\tau(t)}{x+t}\, dt,
   $$
   with a certain positive function $\tau$.
 \end{prop}

 \begin{rem}
  The function $\tau$ is given via boundary values. We shall not give the exact definition; and confine ourselves to showing a plot of it in Figure \ref{fig:tau}.
 \end{rem}

\begin{figure}
\begin{center}
 \includegraphics[scale=0.4]{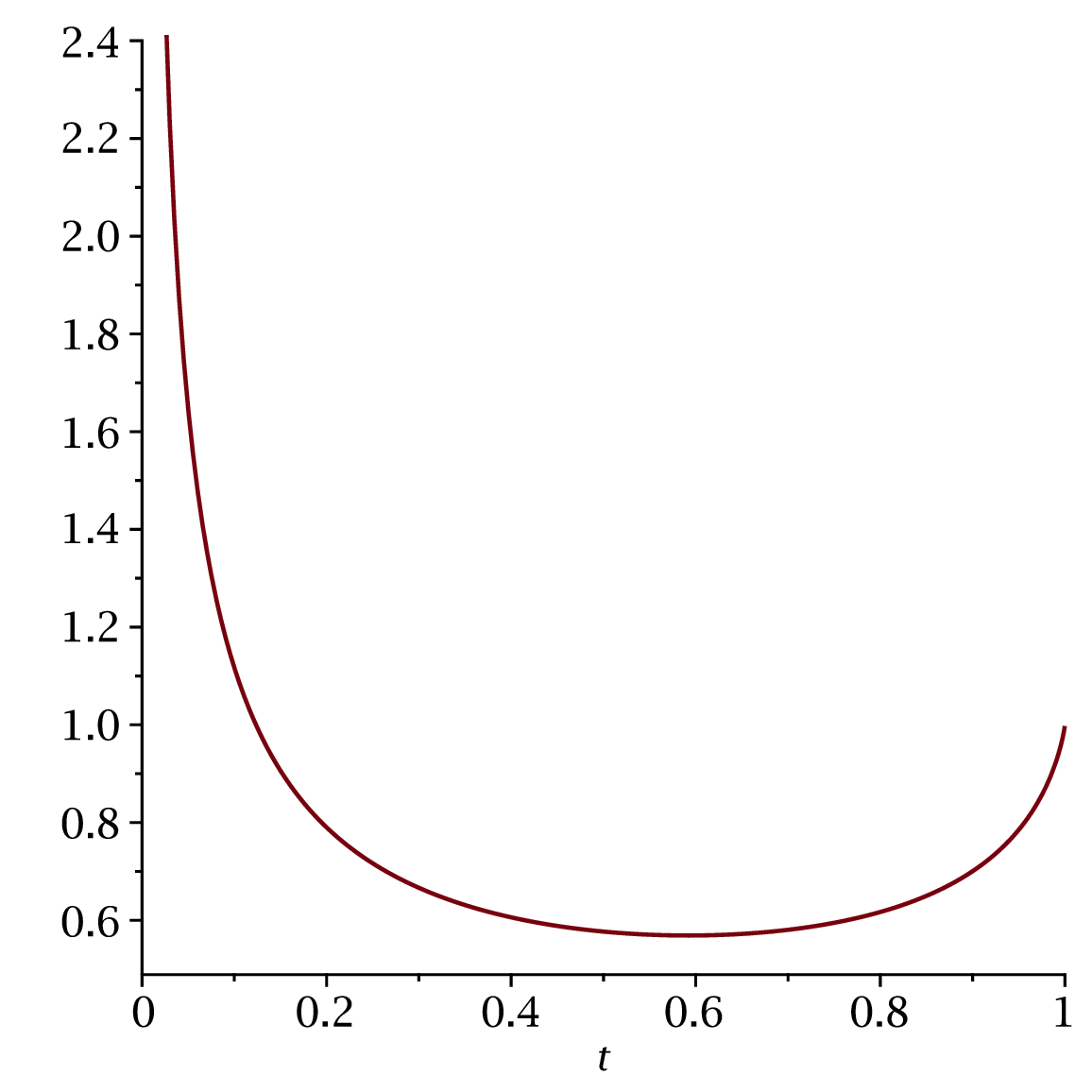}
\end{center}
\caption{ The function $\tau$}
\label{fig:tau}
\end{figure}
Notice $\lim_{x\to \infty}xg(x)=2$ so $\tau$ is a probability measure.

 Since 
 \begin{align*}
   g(x)&=\frac{1}{1+x}+\int_0^1\frac{\tau(s)}{x+s}\, ds\\
   &=\mathcal L(e^{-t})(x)+\int_0^1\tau(s)\int_0^{\infty}e^{-(x+s)t}\, dt \, ds\\
   &=\mathcal L(e^{-t})(x)+\mathcal L\left(\int_0^1\tau(s)e^{-ts}\, ds\right)(x)
   \end{align*}
 we see that $g(x)-a\rho(x)=\mathcal L\left(F_a\right)(x)$, where
 \begin{align*}
   F_a(t)&=(1+a)e^{-t}+\int_0^1\tau(s)e^{-st}\, ds -a\frac{1-e^{-t}}{t}.
 \end{align*}
 Therefore, $g-a\rho\in \text{CM}\Leftrightarrow F_a(t)\geq 0$ for $t\geq 0$. The positivity of $F_a$ is related to the comparison of two moment sequences:

 \begin{align*}
 e^tF_a(t)&=(1+a)+\int_0^1\tau(s)e^{(1-s)t}\, ds -a\frac{e^{t}-1}{t}\\
 &=(1+a)+\int_0^1\tau(1-s)e^{st}\, ds -a\frac{e^{t}-1}{t}\\
 &=(1+a)+\sum_{n=0}^{\infty}\int_0^1\tau(1-s)s^{n}\, ds\frac{t^n}{n!}-a\sum_{n=0}^{\infty}\frac{t^{n}}{(n+1)!},
 \end{align*}
 where $\{t_n\}$ is the moment sequence of the measure $\tau(1-s)\, ds$;
  $$
  t_n=\int_0^1s^n\tau(1-s)\, ds.
  $$
  Therefore we have
  $$
  e^tF_a(t)=2+\sum_{n=1}^{\infty}\left(t_n-\frac{a}{n+1}\right)\frac{t^n}{n!},
  $$
  and this proves the next result.
 \begin{prop}
  $$
  h_a'\in \text{LCM} \ \Leftrightarrow\ 2+\sum_{n=1}^{\infty}\left(t_n-\frac{a}{n+1}\right)\frac{t^n}{n!}\geq 0 \ \ \text{for}\ \ t\geq 0.
  $$
 \end{prop}
The next step is to find the asymptotic behaviour of the moment sequence $\{t_n\}$. With this information one can identify for which $a$ the function $e^tF_a(t)$ is positive. (Observe: it can be positive without all terms in the power series being positive!)

\section{Fifth lecture: Generalized Bernstein functions}
We review the class of Bernstein functions, closely related to completely monotonic functions, and their role in convolution semigroups of probabilities on the positive line. This is followed by the introduction of generalized Bernstein functions of positive order, and their relation to generalized Stieltjes functions of positive order. Examples are given, including functions related to Euler's gamma function.
References: \cite{KP-gen-bernstein-1}, \cite{KP-gen-bernstein-2}.
 \begin{exercise}
  {Recap. Which inclusions hold and which do not?}
  \begin{enumerate}
   \item LCM $\subseteq \text{CM}$
   \item $\mathcal S_{\lambda}\subseteq \text{CM}$
   \item $\mathcal S_{\lambda}\subseteq \text{LCM}$
  \end{enumerate}

 \end{exercise}


\subsection{Bernstein functions}
\begin{defn}
 A function $g:(0,\infty)\to [0,\infty)$ is a Bernstein function if its derivative $g'$ is completely monotonic. This class is denoted by $\mathcal B$.
\end{defn}
These functions admit integral representations, as one would expect considering Bernstein's integral representation of completely monotonic functions.
\begin{prop}
A function $g$ is a Bernstein function if and only if
 $$
 g(x)=a+bx+\int_0^{\infty}(1-e^{xt})\, d\mu(t),
 $$
 where $a\geq 0$, $b\geq 0$ and $\mu$ is a positive measure measure on $(0,\infty)$ satisfying $\int_0^{\infty}td\mu(t)/(t+1)<\infty$.
 \end{prop}
 (We remark that the Bernstein functions are exactly the continuous and so-called conditionally negative definite functions on $(0,\infty)$.)


 \begin{prop}
  For a function $f:(0,\infty)\to \mathbb R$ the following are equivalent:
   \begin{enumerate}
   \item[(i)] $f$ is a Bernstein function;
   \item[(ii)] $f\geq 0$ and $t\mapsto e^{-xf(t)}$ is CM for all $x>0$.
  \end{enumerate}
  \end{prop}
  \begin{proof}
  (i) implies (ii): $-(\log e^{-xf(t)})'=xf'(t)$ is CM, so $e^{-xf(t)}$ is LCM, hence CM.

  (ii) implies (i): We observe that
  $$
  e^{-xf(t)}=\sum_{n=0}^{\infty}\frac{(-1)^nx^n}{n!}f(t)^n.
  $$
  Hence, for $m\geq 1$, and $x,t>0$,
  $$
  0\leq (-1)^m\partial_t^me^{-xf(t)}=\sum_{n=1}^{\infty}\frac{(-1)^{n+m}x^n}{n!}\partial_t^m(f(t)^n).
  $$
  Next divide by $x$, let $x$ tend to $0$ and obtain $(-1)^{m+1}\partial_t^mf(t)\geq 0$.
\end{proof}


 \begin{defn}
A family $\{ \mu_x\}_{x>0}$ of positive measures on
$[0,\infty)$ is called a convolution semigroup on $[0,\infty)$  if it has the properties
\begin{enumerate}
\item[(i)]
$\mu_{x}([0,\infty))\leq  1$ for all $x>0$;
\item[(ii)]
$\mu_x\ast \mu_y=\mu_{x+y}$ for all $x, y>0$;
\item[(iii)]
$\mu_x \to \epsilon_0$ for $x\to 0$ vaguely (recall that $\epsilon_0$ is the point mass at zero).
\end{enumerate}
\end{defn}
\begin{prop}[Bernstein functions and convolution semigroups]
There is a one-to-one correspondence between convolution semigroups on $[0,\infty)$ and the class of Bernstein functions $\mathcal B$:
$\{ \mu_x\}_{x>0}$ corresponds to $f\in \mathcal B$ if
$$
\mathcal L(\mu_x)=e^{-xf}\quad \text{for}\ x>0.
$$
\end{prop}
\begin{proof}
If $f\in \mathcal B$ then $e^{-xf}$ is CM and hence there is  $\mu_x$ such that $\mathcal L(\mu_x)=e^{-xf}$. This gives
$$
\mathcal L(\mu_x\ast\mu_y)=\mathcal L(\mu_x)\mathcal L(\mu_y)=e^{-xf}e^{-yf}=e^{-(x+y)f}=\mathcal L(\mu_{x+y}).
$$
By uniqueness in Bernstein's theorem, $\mu_x\ast\mu_y=\mu_{x+y}$ and hence $\{\mu_x\}_{x>0}$ has the second property in the definition of a convolution semigroup above.
\end{proof}

\subsection{Generalized Bernstein functions of positive order}

%
 
 \begin{defn}
  Let $\lambda >0$. A function $g:(0,\infty)\to [0,\infty)$ is called a generalized Bernstein function of order $\lambda$ if $g$ is $C^{\infty}$ and $g'(x)x^{1-\lambda}$ is CM. This class is denoted by $\mathcal B_{\lambda}$.
 \end{defn}
 \begin{defn}[Incomplete Gamma function] The incomplete Gamma function $\gamma:(0,\infty)\times (0,\infty)\to (0,\infty)$ is defined as
  $$
 \gamma(\lambda,x)=\int_0^{x}u^{\lambda-1}e^{-u}\, du.
 $$
 \end{defn}
In particular $\gamma(1,x)=1-e^{x}$, which is the expression appearing in the integral representation of Bernstein functions. It is not surprising that there is an integral representation of generalized Bernstein functions in terms of the incomplete gamma function.
 \begin{prop}[Integral representation]
  A function $g$ belongs to $\mathcal B_{\lambda}$ if and only if there exist $a\geq 0$, $b\geq 0$ and a positive measure measure $\mu$ on $(0,\infty)$ satisfying $\int_0^{\infty}\mu(t)/(t+1)^{\lambda}<\infty$, such that
 \begin{align*}
 g(x)
 &=a+bx^{\lambda}+\int_0^{\infty}\gamma(\lambda,xt)\, \frac{d\mu(t)}{t^{\lambda}}.
 \end{align*}
 \end{prop}

  We saw earlier that $f\in \mathcal S_{\lambda}$ if and only if  $c_k(f)(x)\equiv x^{1-\lambda}(x^{\lambda-1+k}f(x))^{(k)}$ is completely monotonic for all $k\geq 0$. Since $c_1(f)(x)=x^{1-\lambda}(x^{\lambda}f'(x))$, this led us to study functions for which $x^{1-\lambda}$ times ``its derivative'' is CM.
The definition of generalized Bernstein functions is motivated by this.

  \begin{exercise}
   Show that $\mathcal B_{\lambda_1}\subset \mathcal B_{\lambda_2}$ when $\lambda_1<\lambda_2$.
  \end{exercise}

  \begin{exercise}
  Give a partial proof of the integral representation above: Show that if $g\in \mathcal B_{\lambda}$ then
  \begin{align*}
 g(x)
 &=a+bx^{\lambda}+\int_0^{\infty}\gamma(\lambda,xt)\, \frac{d\mu(t)}{t^{\lambda}}.
 \end{align*}
  for some $a\geq 0$, $b\geq 0$ and for some positive measure measure $\mu$ on $(0,\infty)$ satisfying $\int_0^{\infty}\mu(t)/(t+1)^{\lambda}<\infty$.
 \end{exercise}
 \begin{exercise} Prove the following:
  \begin{enumerate}
   \item If $g\in \mathcal B_{\lambda}$ then $g(x)/x^{\lambda}$ is CM.
   \item If $g_1\in \mathcal B_{\lambda_1}$ and $g_2\in \mathcal B_{\lambda_2}$ then $g_1g_2\in \mathcal B_{\lambda_1+\lambda_2}$
  \end{enumerate}

 \end{exercise}

 Let us end this section by giving some examples of concrete special functions from $\mathcal B_{\lambda}$.
 \begin{enumerate}
  \item The incomplete gamma function $\gamma(\lambda,x)$ belongs to $\mathcal B_{\lambda}$.
  \item Lerch's transcendent
 $$
 \Phi(z,1,\lambda)=\sum_{n=0}^{\infty}\frac{z^n}{n+\lambda}, \quad |z|<1,
 $$
 satisfies
  $$
 x^{\lambda}\Phi(-x,1,\lambda)=\int_0^{x}\frac{u^{\lambda-1}}{1+u}\, du.
 $$
 Hence $x^{\lambda}\Phi(-x,1,\lambda)\in \mathcal B_{\lambda}$.
\item The hypergeometric function
$$
{}_2F_1(\nu,\lambda; 1+\lambda;-x)=\sum_{n=0}^{\infty}\frac{(\nu)_n(\lambda)_n}{n!(1+\lambda)_n}(-1)^nz^n
$$
satisfies (for $\nu>0$)
$$
 x^{\lambda}{}_2F_1(\nu,\lambda; 1+\lambda;-x)=\lambda \int_0^x\frac{u^{\lambda-1}}{(1+u)^\nu}du.
 $$
 Hence $x^{\lambda}{}_2F_1(\nu,\lambda; 1+\lambda;-x)\in \mathcal B_{\lambda}$.
 \item The function
 $g_{\lambda}(x)=x^{\lambda}\Gamma(x)/\Gamma(x+\lambda)$ belongs to $ \mathcal B_{\lambda-1}$ for $\lambda >1$. (More on this later.)

 \end{enumerate}

\subsection{The hierarchy of generalized Bernstein functions}

Recall that a function $f$ is a generalized Stieltjes function of order $\lambda$ ($f\in \mathcal S_{\lambda}$) if
 $$
 f(x)=c+\int_0^{\infty}\frac{d\mu(t)}{(t+x)^{\lambda}}, \ x>0.
 $$
As our first result a bijection between this class and the class of generalized Bernstein functions of order $\lambda$ is given.

 \begin{defn}
 Define $\XL: \mathcal B_{\lambda}\to C^{\infty}((0,\infty))$ by
$$
\XL(g)(x)=x\mathcal L(g)(x)=x\int_0^{\infty}e^{-xt}g(t)\, dt.
$$
\end{defn}

\begin{prop}The mapping
 $\XL$ is a bijection of $\mathcal B_{\lambda}$ onto $\mathcal S_{\lambda}$.
\end{prop}

We now turn to properties of the entire family $\{\mathcal B_{\lambda}\}_{\lambda>0}$. From an exercise we know that if $\lambda_1<\lambda_2$ then $\mathcal B_{\lambda_1}\subset \mathcal B_{\lambda_2}$, so the class $\mathcal B_{\lambda}$ gets bigger with $\lambda$. In the following two results we identify $\mathcal B_0=\bigcap_{\lambda>0}\mathcal B_{\lambda}$ and $\mathcal B_{\infty}=\bigcup_{\lambda>0}\mathcal B_{\lambda}$.

\begin{center}
 \includegraphics[scale=0.3]{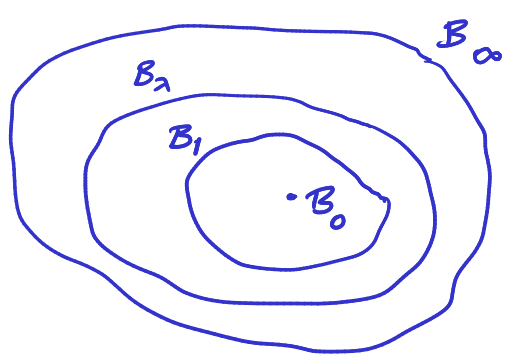}
\end{center}

\begin{prop}
$$
\mathcal B_{0}= \{ \text{the constant functions with values in} \ [0,\infty)\}
$$
\end{prop}

\begin{prop}
 Let $\mu$ be a Radon measure on $[0,\infty)$. Then there is a sequence $\{b_n\}$ of functions from  $\mathcal B_{\infty}$ such that
$$
b_n'(x)\, dx \to \mu\
$$
vaguely as $n\to \infty$.

\end{prop}
The proof of this result uses the following lemma.
\begin{prop}
 Let $g_n(x)=n^nx^{n-1}e^{-nx}/\Gamma(n)$.  (The Gamma distribution  with shape parameter $n$ and  scale parameter $1/n$.) Then
 $\{g_n(x)\, dx\}$ converges weak-star to the point mass at $1$.
\end{prop}
\subsection{Higher order Thorin-Bernstein functions }

 A Bernstein function $g$ of the form
 $$
 g(x)=a+bx+\int_0^{\infty}(1-e^{-xt})\varphi(t)\, dt
$$
is called a
 \begin{itemize}
  \item complete Bernstein function if $\varphi$ is CM
  \item Thorin-Bernstein function if $t\varphi(t)$ is CM
  \end{itemize}
  These subclasses of ordinary Bernstein functions are used in probability theory.
Our aim is to study them in the framework of generalized Bernstein functions of positive order.

Let us recall that
  $\varphi$ is $\text{CM}(\alpha)$ if $t^{\alpha}\varphi(t)$ is CM.

\begin{defn}
 A function $f:(0,\infty)\to \mathbb R$ is a $(\lambda,\alpha)$-Thorin-Bernstein function, if $\lambda>0$ and $\alpha<\lambda+1$ are such that
 \begin{equation*}
  f(x)=a x^{\lambda}+b+\int_0^{\infty}\gamma(\lambda,xt)\varphi(t)\, dt,
\end{equation*}
where $a$ and $b$ are non-negative numbers, and $\varphi$ is CM($\alpha$). The class of these functions is denoted by $\mathcal T_{\lambda,\alpha}$.
\end{defn}
We shall call these classes of functions {\em higher order Thorin-Bernstein functions}.
The complete Bernstein functions mentioned above correspond to $\lambda=1$, $\alpha=0$ and the Thorin-Bernstein functions correspond to $\lambda=1$, $\alpha=1$.

There is a close connection between higher order Thorin-Bernstein functions and generalized Stieltjes functions of positive order:

\begin{prop}
 For a function $f:(0,\infty)\to [0,\infty)$ we have
 $$ f\in \mathcal T_{\lambda,\alpha} \ \Leftrightarrow \ x^{1-\lambda}f'(x)\in \mathcal S_{\lambda+1-\alpha}.
 $$
\end{prop}

Returning to the example above about ${}_2F_1$ we see that in fact
$$
 x^{\lambda}{}_2F_1(\nu,\lambda; 1+\lambda;-x)=\lambda \int_0^x\frac{u^{\lambda-1}}{(1+u)^\nu}du\in \mathcal T_{\lambda, \lambda+1-\nu}.
 $$
We end this section by investigating the family $\{\mathcal T_{\lambda,\alpha}\}_{\alpha<\lambda +1}$. By definition, $\mathcal T_{\lambda,\alpha}\subset \mathcal B_{\lambda}$ for $\alpha<\lambda+1$, and it is also easy to show that
$T_{\lambda,\alpha}$ decreases as $\alpha$ increases.
The question of the ``size'' of $\cup_{\alpha<\lambda +1}\mathcal T_{\lambda,\alpha}$ is answered below.

The incomplete Beta function is defined as
$$
B(a,b,x)=\int_0^xt^{a-1}(1-t)^{b-1}\, dt
$$
for $a>0$, $b\in \mathbb R$ and $x\in [0,1)$.

\begin{prop}
Let $f\in \mathcal B_{\lambda}$ correspond to the triple $(a,b,\mu)$. Then there is a sequence from $\cup_{\alpha<\lambda +1}\mathcal T_{\lambda,\alpha}$ that converges pointwise to $f$. One may in fact take the sequence $\{f_n\}$ to be
$$
f_n(x)=ax^{\lambda}+b+\frac{\Gamma(\lambda+n)}{\Gamma(n)}\,\int_0^{\infty}  B\Big(\lambda, n, \frac{x}{x+n/t}\Big)\,\frac{d\mu(t)}{t^{\lambda}},
$$
where $B$ is the incomplete Beta function.
\end{prop}

\subsection{Returning to Euler's gamma function}
We shall end these lectures where we began, namely investigating the Gamma function. This time we shall illustrate the real variable methods outlined above.
 Many authors (going back to Ismail and others in the 1980's) have investigated ratios of gamma functions, such as
 $$
 \frac{\Gamma(x)}{\Gamma(x+a)}\quad \text{or}\quad \frac{\Gamma(x)\Gamma(x+a+b)}{\Gamma(x+a)\Gamma(x+b)}.
 $$
 The $\psi$-function, defined as
 $$
 \psi(x)=(\log \Gamma)'(x)=\frac{\Gamma'(x)}{\Gamma(x)}
 $$
and its integral representation
$$
\psi(x)=-\gamma+\int_0^{\infty}\frac{e^{-t}-e^{-xt}}{1-e^{-t}}\, dt \quad \text{for }\ x>0
$$
play important roles in most of these investigations.

  A computation, using the integral representation of $\psi$, gives ($a$ and $b$ being positive numbers)
  $$
  \log\frac{\Gamma(x)\Gamma(x+a+b)}{\Gamma(x+a)\Gamma(x+b)}
  =\int_0^{\infty}te^{-xt}\frac{(1-e^{-at})(1-e^{-bt})}{t^2(1-e^{-t})}\, dt.
  $$
  The elementary computation
  $
  (1-e^{-at})/t=\int_0^ae^{-st}\, ds
  $ shows that $(1-e^{-at})/t$ is CM and hence the product
  $$
  \frac{(1-e^{-at})(1-e^{-bt})}{t^2(1-e^{-t})}=\frac{1-e^{-at}}{t}\frac{1-e^{-bt}}{t}\frac{1}{1-e^{-t}}
  $$
  is CM. Referring to the relation above and Kristiansen's theorem we obtain the proof of the following result.
 \begin{prop}
  The function
  $$
  x\mapsto \log\frac{\Gamma(x)\Gamma(x+a+b)}{\Gamma(x+a)\Gamma(x+b)}
  $$
  belongs to $\mathcal S_2$ and is in particular LCM.
  \end{prop}

We shall investigate a similar function $g_{\lambda}$ given by
$$
 g_{\lambda}(x)=\frac{x^{\lambda}\Gamma(x)}{\Gamma(x+\lambda)}.
 $$

  \begin{figure}
  \begin{center}
  \includegraphics[scale=0.5]{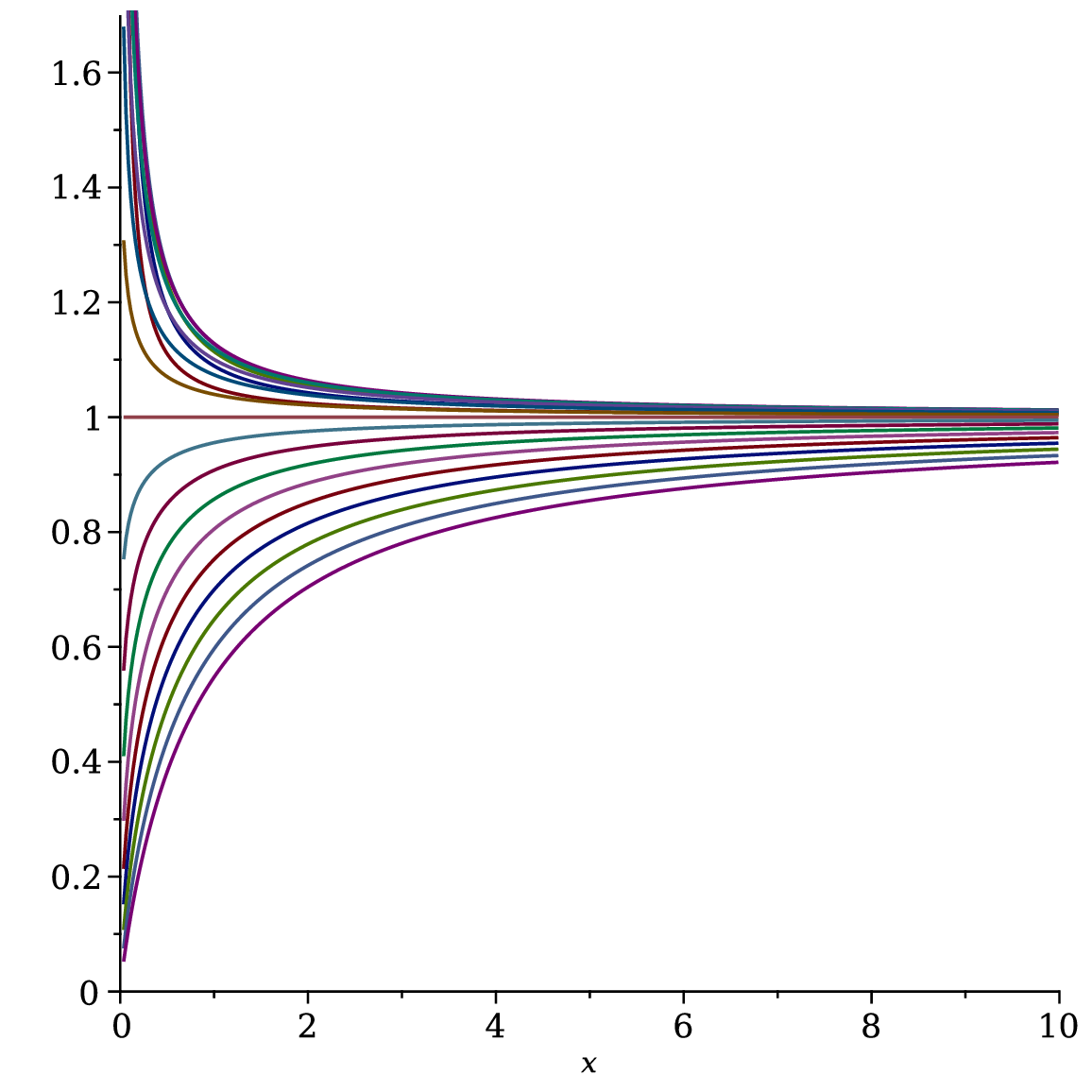}
 \end{center}
 \caption{$g_{\lambda}$ for various values of $\lambda$}
 \label{fig:glambda}
 \end{figure}

 Referring to Figure \ref{fig:glambda}, the function is decreasing for $\lambda<1$, and increasing for $\lambda>1$. The functions $g_{\lambda}$ were also considered by Ismail, Lorch and Muldoon who showed that $g_{\lambda}$ is LCM for $\lambda<1$. We shall strengthen this result for $\lambda<1$ and give analogues for $\lambda>1$.

 \begin{prop} The following assertions  hold.
  \begin{enumerate}
   \item For $\lambda<1$, $\log g_{\lambda}\in \mathcal S_2\setminus \cup_{\tau<2}\mathcal{S}_{\tau}$.
   \item For $\lambda>1$, $g_{\lambda}\in \mathcal{B}_{\lambda-1}\setminus \cup_{\tau<\lambda-1}\mathcal{B}_{\tau}$.
   \item For $\lambda>1$, $-\log g_{\lambda}\in \mathcal{S}_2\setminus \cup_{\tau<2}\mathcal{S}_{\tau}$.
  \end{enumerate}
\end{prop}
\begin{proof}
We shall only indicate the proof of a part of the second assertion. We consider
\begin{align*}
 \sigma_{\lambda}(x)=g_{\lambda}'(x)/g_{\lambda}(x)&=(\log g_{\lambda})'(x)=\lambda/x+\psi(x)-\psi(x+\lambda),
 \end{align*}
 and observe that
 \begin{align*}
 \sigma_{\lambda}(x)&=\int_0^{\infty}e^{-xt}\xi(t)\, dt,
 \end{align*}
 where
 $\xi(t)=\lambda-(1-e^{-\lambda t})/(1-e^{-t})$ is positive and increasing. This means that $\sigma_{\lambda}\in \text{CM(1)}$ so that
 \begin{align*}
 x^{2-\lambda}g_{\lambda}'(x)&=x^{1-\lambda}g_{\lambda}(x)\, x\sigma_{\lambda}(x)= \underset{\text{CM}}{\underbrace{\tfrac{\Gamma(x+1)}{\Gamma(x+\lambda)}}}\, \, \underset{\text{CM}}{\underbrace{x\sigma_{\lambda}(x)}}.
\end{align*}
Thus $g_{\lambda}$ belongs to $\mathcal B_{\lambda-1}$.
\end{proof}

\section*{Conclusion}
These lectures consisted of an investigation of classes of holomorphic functions and their relation to classes of functions on the positive real line.

We have demonstrated how complex methods can be used to solve questions for functions on the real line, in particular for the Gamma and related functions. Positivity has played an important part in obtaining integral representations of special functions.

\bibliographystyle{amsplain}

\bibliography{hlp-references}
 \end{document}